\documentclass[12pt]{article}
\usepackage{graphicx}
\usepackage{amscd, amssymb, amsmath}
\usepackage{enumerate}
\usepackage{hyperref}
\usepackage{lscape}
\newenvironment{eq}{\begin{equation}}{\end{equation}}
\newenvironment{proof}{{\bf Proof}:}{\vskip 5mm }

\newtheorem{proposition}{Proposition}[section]
\newtheorem{lemma}[proposition]{Lemma}
\newtheorem{definition}[proposition]{Definition}

\newtheorem{example}[proposition]{Example}
\newtheorem{remark}[proposition]{Remark}

\newtheorem{problem}[proposition]{Problem}
\newtheorem{construction}[proposition]{Construction}
\newcommand{\llabel}[1]{\label{#1}}
\newcommand{\comment}[1]{}
\newcommand{\sr}{\rightarrow}

\newcommand{\nn}{{\bf N\rm}}

\newcommand{\uu}{\underline}

\newcommand{\wt}{\widetilde}

\newcommand{\dd}{\diamond}

{\vskip
3mm}

\newcommand{\spc}{{\,\,\,\,\,\,\,}}

\newcommand{\wtOb}{{\wt{\mathcal Ob}}}
\newcommand{\Ob}{{\mathcal Ob}}
\newcommand{\wOb}{\wt{\Ob}}

\newcommand{\nat}{\nn}

\begin{document}
\parskip = 2mm
\begin{center}
{\bf\Large C-systems defined by universe categories: presheaves}{$^,$}\footnote{For the published version of the paper see \url{http://www.tac.mta.ca/tac/volumes/32/3/32-03abs.html}}

\vspace{3mm}

{\large\bf Vladimir Voevodsky}\footnote{School of Mathematics, Institute for Advanced Study,
Princeton NJ, USA. e-mail: vladimir@ias.edu}
\vspace {3mm}

\end{center}
\begin{abstract}
The main result of this paper may be stated as a construction of ``almost representations'' $\mu_n$ and $\wt{\mu}_n$ for the presheaves $\Ob_n$ and $\wOb_n$ on the C-systems defined by locally cartesian closed universe categories with binary product structures and the study of the behavior of these ``almost representations'' with respect to the universe category functors. 

In addition, we study a number of constructions on presheaves on C-systems and on universe categories that are used in the proofs of our main results, but are expected to have other applications as well. 
\end{abstract}

\tableofcontents

\numberwithin{equation}{section}

\section{Introduction}
\label{Sec.0}

The concept of a C-system in its present form was introduced in \cite{Csubsystems}. The type of the C-systems is constructively equivalent to the type of contextual categories defined by Cartmell in \cite{Cartmell0} and \cite{Cartmell1} but the definition of a C-system is slightly different from Cartmell's foundational definition.

In \cite[Sec.3]{fromunivwithPiI} we constructed, on any C-system and for any $n\ge 0$,  presheaves $\Ob_n$ and $\wtOb_n$. These presheaves play a major role in our approach to the C-system formulation of systems of operations that correspond to systems of inference rules. For example, by providing a construction for \cite[Prob. 4.5]{fromunivwithPiI} we construct a representation in terms of these presheaves of the most important structure on the C-systems, the structure corresponding to the $(\Pi,\lambda, app, \beta, \eta)$-system of inference rules. Another formulation of this structure is the Cartmell-Streicher structure of products of families of types \cite[pp. 3.37 and 3.41]{Cartmell0}, \cite[p. 71]{Streicher}. 

This paper appeared as an outcome of the work on systematization and generalization of the ideas that has been used in the second part of the preprint \cite{fromunivwithPi}. In the process of this work new structures became apparent both on the C-systems and on the universe categories whose general properties could be combined to obtained clear and systematic proofs of the main theorems of the second part of that preprint. 

As a result the paper is about certain constructions on presheaves on C-systems and universe categories, about their interaction in the case of the C-systems of the form $CC({\cal C},p)$, and about their behavior with respect to the universe category functors. 

The paper is subdivided into two parts. In the first part there are collected constructions and theorems (most often called `lemmas')  about structures on one C-system or universe category or one pair of a C-system and a universe category. In the second part we study the behavior of these constructions with respect to the universe category functors.   

Throughout the paper we use the diagrammatic order in writing the composition, that is, for $f:X\sr Y$ and $g:Y\sr Z$ we write $f\circ g$ for their composition. This convention applies to functions between sets, morphisms in categories, functors etc. 

We denote by $\Phi^{\circ}$ the functor $PreShv(C')\sr PreShv(C)$ given by the pre-composition with a functor $\Phi^{op}:C^{op}\sr (C')^{op}$, that is, 
$$\Phi^{\circ}(F)(X)=F(\Phi(X))$$
In the literature this functor is denoted both by $\Phi^*$ and $\Phi_*$ and we decided to use a new unambiguous notation instead. 

In the first Section \ref{Sec.1} we study a functor $Sig:PreShv(CC)\sr PreShv(CC)$ defined from any C-system $CC$. This functor comes together with functor isomorphisms 
$$S\Ob_n: Sig(\Ob_n)\sr \Ob_{n+1}$$
$$S\wOb_n:Sig(\wOb_n)\sr \wOb_{n+1}$$
that makes it important for the study of the `structure sheaves' $\Ob_n$ and $\wOb_n$.

In Section \ref{Sec.2} we study a functor $D_p:PreShv({\cal C})\sr PreShv({\cal C})$ that is defined for any category $\cal C$ with a choice of a universe $p$ (see \cite[Def. 2.1]{Cfromauniverse}) in it. We define the sets $D_p^n(X,Y)$, which play a major role in all that follows, by the formula
$$D_p^n(X,Y)=D_p^n(Yo(Y))(X)$$
where $Yo(Y)$ is the contravariant functor represented by $Y$ and $D_p^n$ refers to the $n$-th iteration of $D_p$. These sets are functorial both in $X$ and $Y$ and we use the notation $\circ$ for this functoriality, that is, for $d\in D_p^n(X,Y)$, $f:X'\sr X$ and $g:Y\sr Y'$ we set
$$f\circ d=D_p^n(f,Y)(d)$$
$$d\circ g=D_p^n(X,g)(g)$$
This $\circ$-notation is justified by the identity and associativity equalities collected in Lemma \ref{2017.01.07.l1} and makes many computations in the following sections more convenient. 

In Section \ref{Sec.3} we construct the first of the main isomorphisms that we work with in this paper, namely, the isomorphisms $u_1$ and $\wt{u}_1$ of the form
$$u_1:\Ob_1\sr int^{\circ}(Yo(U))$$
$$\wt{u}_1:\wOb_1\sr int^{\circ}(Yo(\wt{U}))$$
Here we start with a universe category, a category with a universe and a final object, that we denote $({\cal C},p)$ where $p:\wt{U}\sr U$ is the universe. Associated to it in \cite[Constr. 2.12]{Cfromauniverse} is a C-system $CC({\cal C},p)$ that comes together with a fully-faithful functor $int:CC({\cal C},p)\sr {\cal C}$. For a functor $\Phi$ we let $\Phi^{\circ}$ denote the functor of pre-composition with $\Phi$ on presheaves. In particular, for a presheaf $F$ on ${\cal C}$, $int^{\circ}(F)$ is a presheaf on $CC({\cal C},p)$. The isomorphisms $u_1$ and $\wt{u}_1$ provide what we referred to as `almost representations' for $\Ob_1$ and $\wOb_1$.

In Section \ref{Sec.4} we connect, by a functor isomorphism $SD_p$,  the functors $D_p$ and $Sig$ in the case of a pair of a universe category and the associated C-system. The isomorphism $SD_p$ has the form
$$SD_p:int^{\circ}\circ Sig\sr D_p\circ int^{\circ}$$

In Section \ref{Sec.5} we combine $S\Ob_n$, $Sig$ and $SD_p$ to provide the inductive step for a definition of isomorphisms 
$$u_n:\Ob_n\sr int^{\circ}(D_p^{n-1}(Yo(U)))$$
$$\wt{u}_n:\wOb_n\sr int^{\circ}(D_p^{n-1}(Yo(\wt{U})))$$
starting with $u_1$ and $\wt{u}_1$. This construction combines all constructions introduced in Sections \ref{Sec.1}-\ref{Sec.4}. 

Section \ref{Sec.6} is different from the preceding sections in that, that we assume an additional structure on $\cal C$ namely the combination of the locally cartesian closed and binary product structures. Under the assumption of these structures we construct a functor $I_p:{\cal C}\sr {\cal C}$ and a family of representations for presheaves of the form $D_p(Yo(V))$ 
that we denote by $\eta_V$:
$$\eta_V:D_p(Yo(V))\sr Yo(I_p(V))$$
This family is natural in $V$ forming a functor isomorphism
$$\eta:Yo\circ D_p\sr I_p\circ Yo$$
where both functors are from $\cal C$ to $PreShv({\cal C})$. By a simple iterative construction we obtain from $\eta$ functor isomorphisms
$$\eta_n:Yo\circ D_p^n\sr I_p^n\circ Yo$$
and with them representations of presheaves of the form $D_p^n(Yo(V))$. The naturality properties of our construction admit convenient presentation in the $\circ$-notation:
$$\eta_n(f\circ d)=f\circ \eta_n(d)$$
$$\eta_n(d\circ g)=\eta_n(d)\circ I_p^n(g)$$
Combining isomorphisms $\eta$ with isomorphisms $u$ and $\wt{u}$ we obtain the second main family of isomorphisms constructed in this paper
$$\mu_n:\Ob_n\sr int^{\circ}(Yo(I_p^{n-1}(U)))$$
$$\wt{\mu}_n:\wOb_n \sr int^{\circ}(Yo(I_p^{n-1}(\wt{U})))$$

To be able to rigorously construct $I_p$ as a functor and later to study its properties with respect to universe category functors we need notations related to the locally cartesian closed structures that we could not find in the literature. Therefore, we added the two appendices where these notations are introduced. The appendices contain known facts about categories, but it seems that these facts have never been represented at the level of detail that we need in this paper. 

We may add that using a locally cartesian closed and a binary product structure to construct $I_p(V)$ and the representations of the functors $D_p(Yo(V))$ by means of these objects is to require a lot where much less would suffice. It would be sufficient to simply require a family of objects $I_p(V)$ parametrized by $V\in {\cal C}$ and representations $\eta_V$ of $D_p(Yo(V))$ by means of $I_p(V)$. The functoriality for this family in $V$ may then be derived from the representations and the functoriality of $Yo\circ D_p$ and this functoriality will automatically make the representations natural in $V$. This is a much less grand requirement that a full locally cartesian closed plus a binary products structure.  However, we will continue to consider the construction based on these two structures because they are more familiar. 

Section \ref{Sec.6} completes the first part of the paper. In the second part we consider how the constructions introduced in the first part interact with the universe category functors. 

In Section \ref{Sec.7} we study the interaction of the functor $D_p$ and the construction $D_p^n(X,Y)$ with the universe category functors. We start with a universe category functor ${\bf\Phi}=(\Phi,\phi,\wt{\phi}):({\cal C},p)\sr ({\cal C}',p')$ and construct a functor isomorphism
$$\Phi D:\Phi^{\circ}\circ D_p\sr D_{p'}\circ \Phi^{\circ}$$
Then we use this isomorphism along with a number of other constructions to define, for all $X,Y\in {\cal C}$ and $n\ge 0$, functions
$${\bf\Phi}^n_{X,Y}:D_p^n(X,Y)\sr D_{p'}^n(\Phi(X),\Phi(Y))$$
These functions are natural in $X$ and $Y$, which, in the $\circ$-notation, can be written as
$${\bf\Phi}^n(f\circ d)=\Phi(f)\circ {\bf\Phi}^n(d)$$
$${\bf\Phi}^n(d\circ g)={\bf\Phi}^n(d)\circ \Phi(g)$$

In Section \ref{Sec.8} we study the interaction of the isomorphisms $u_n$ and $\wt{u}_n$ with the universe category functors. First, we recall the construction, given in \cite[Constr. 4.7]{Cfromauniverse}, of the homomorphism of C-systems $H:CC({\cal C},p)\sr CC({\cal C}',p')$ defined by $\bf\Phi$. We also show that the family of isomorphisms $\psi(\Gamma)$ constructed in the same paper forms a functor  isomorphism
$$\psi:H\circ int\sr int\circ \Phi$$
The main result of this section is Lemma \ref{2016.12.20.l4} that establishes the rule under which isomorphisms $u_n$ and $\wt{u}_n$ are transformed by universe category functors. 

In Section \ref{Sec.9} we study the interaction of universe category functors between locally cartesian closed universe categories with binary product structures and the constructions $I_p$ and $\eta_n$. No assumption about the compatibility of the functor with the locally cartesian closed or binary product structures used to define $I_p$ and $I_{p'}$ and the corresponding $\eta_n$ functions is made. In a somewhat surprising result no such assumption turns out to be necessary to construct a natural family of morphisms 
$$\chi_{{\bf\Phi},n}:\Phi(I_p^n(Y))\sr I_{p'}^n(\Phi(Y))$$
that satisfy, for $d\in {\bf\Phi}^n(X,Y)$, the equalities
$$\eta_n({\bf\Phi}^n(d))=\Phi(\eta_n(d))\circ \chi_n(Y)$$

In Section \ref{Sec.10} we apply the results of the previous section and Section \ref{Sec.8} to describe the interaction of isomorphisms $\mu_n$ and $\wt{\mu}_n$ with universe category functors. This is the last result of the present paper. 

In \cite{fromunivwithPiII} the result of Section \ref{Sec.10} it is used to prove the functoriality theorem for the construction of a $(\Pi,\lambda)$-structure on $CC({\cal C},p)$ from a $P$-structure on $p$. 

Throughout the paper we work in the Zermelo-Fraenkel foundations. The methods of this paper are constructive in the following sense. Neither the axiom of the excluded middle nor the axiom of choice are used. Unbounded universal quantification is used since we make statements about ``any C-system'' etc. and operate with concepts such as a presheaf or a family without specifying the target universe. I seems that this use of unbounded quantification can be easily eliminated at the price of extending the foundational system to include one or two universes. We have not fully analyzed the requirements that these universes would have to satisfy. At the moment it appears that it would be sufficient to use ``small'' universes whose existence can be proved in ZF. 

At the same time as complying with the requirements imposed by ZF we made an effort to ensure that the paper can be translated into the UniMath language. 

The paper is written in the formalization-ready style, that is, in such a way that no long arguments are hidden even when they are required only to substantiate an assertion that may feel obvious to readers who are closely associated with a particular tradition of mathematical thought. 

As a result, a number of lemmas, especially in the appendices, may be well know to many readers. Their proofs are nevertheless included to comply with the requirements of the formalization ready style. 

On the other hand, not all preliminary lemmas are included or a reference to a complete proof is given. There are some, but very much fewer than is usual in today's papers, exceptions. 

The concept of ``a family'' can be formalized as suggested in \cite[Remark 3.9]{fromunivwithPiI}. 

The main results of this paper are not theorems but constructions and so are many of the intermediate results. Because of the importance of constructions for this paper we use a special pair of names Problem-Construction for the specification of the goal of a construction and the description of the particular solution.

In the case of a Theorem-Proof pair one usually refers (by name or number) to the theorem when  using the proof of this theorem. This is acceptable in the case of theorems because the future use of their proofs is such that only the fact that there is a proof but not the particulars of the proof matter. 

In the case of a Problem-Construction pair the content of the construction often matters in the future use. Because of this we have to refer to the construction and not to the problem and we assign in this paper numbers both to Problems and to Constructions. 

Acknowledgements are at the end of the paper. 

\newpage

\section{Presheaves $\Ob_n$ and $\wt{\Ob_n}$}

\subsection{Functor $Sig$ and functor isomorphisms $S\Ob_n$ and $S\wt{Ob}_n$}
\llabel{Sec.1}
Let $CC$ be a C-system. Presheaves $\Ob_n$ and $\wOb_n$ on $CC$ where defined in \cite[Section 3]{fromunivwithPiI}. On objects they are given by
\begin{eq}
\llabel{2016.11.15.eq5}
\Ob_n(\Gamma)=\{T\in Ob(CC)\,|\,l(T)=l(\Gamma)+n,\, ft^n(T)=\Gamma\}
\end{eq}
\begin{eq}
\llabel{2016.11.15.eq6}
\wOb_n(\Gamma)=\{o\in Mor(CC)\,|\,codom(o)\in\Ob_n(\Gamma),o\in sec(p_{codom(o)}),codom(o)>\Gamma\}
\end{eq}
where for a morphism $p:Y\sr X$ we set
$$sec(p)=\{s\in Mor(X,Y)\,|\,s\circ p=Id_X\}$$
Elements of $sec(p)$ are called sections of $p$. 

On morphisms these presheaves are defined by 
$$\Ob_n(f)(T)=f^*(T)$$
$$\wOb_n(o)=f^*(o)$$
where $f^*(T)$ is defined in \cite[above Lemma 2.4]{fromunivwithPiI} and $f^*(o)$ is defined in  \cite[Lemma 2.13]{fromunivwithPiI}.

For $o\in\wOb_n(\Gamma)$ we write $\partial_{n,\Gamma}(o)$ for $codom(o)$. We will often omit the indexes $n$ and $\Gamma$ at $\partial$. The family of functions $\partial_{n,\Gamma}$ forms a morphism of presheaves $\partial_n:\wOb_n\sr \Ob_n$. 

In this section we consider three constructions that apply to any C-system $CC$ -- a functor $Sig:PreShv(CC)\sr PreShv(CC)$ and two families of isomorphisms paramerized by $n\in\nat$:
$$S\Ob_n:Sig(\Ob_n)\sr \Ob_{n+1}$$
and
$$S\wt{\Ob}_n:Sig(\wt{\Ob}_n)\sr \wt{Ob}_{n+1}$$
such that $S\wOb_n\circ \partial_{n+1}=Sig(\partial_n)\circ S\Ob_n$.

Let ${\cal G}$ be a presheaf on $CC$. For $\Gamma\in CC$ we set
\begin{eq}
\llabel{2016.08.30.eq7}
Sig({\cal G})(\Gamma)=\amalg_{T\in Ob_1(\Gamma)}{\cal G}(T)
\end{eq}
and for $f:\Gamma'\sr \Gamma$ 
\begin{eq}
\llabel{2016.08.30.eq8}
Sig({\cal G})(f)(T,g)=(f^*(T),{\cal G}(q(f,T))(T))
\end{eq}
\begin{lemma}
\llabel{2016.08.28.l1}
The presheaf data $Sig$ is a presheaf, that is, one has:
\begin{enumerate}
\item for $\Gamma\in CC$, 
$$Sig({\cal G})(Id_{\Gamma})=Id_{Sig({\cal G})(\Gamma)}$$
\item for $f':\Gamma''\sr \Gamma'$, $f:\Gamma'\sr \Gamma$,
$$Sig({\cal G})(f'\circ f)=Sig({\cal G})(f)\circ Sig({\cal G})(f')$$
\end{enumerate}
\end{lemma}
\begin{proof}
For the identity we have
$$Sig({\cal G})(Id_{\Gamma})(T,g)=(Id_{\Gamma}^*(T), {\cal G}(q(Id_{\Gamma},T))(g))=(T,g)$$
where the second equality is by axioms of the C-system structure. For the composition we have
$$Sig({\cal G})(f')(Sig({\cal G}(f)(T,g)))=
Sig({\cal G})(f')(f^*(T), {\cal G}(q(f,T))(g))=$$$$
((f')^*(f^*(T)), {\cal G}(q(f',f^*(T)))({\cal G}(q(f,T))(g)))=
((f')^*(f^*(T)), {\cal G}(q(f',f^*(T))\circ q(f,T))(g))=$$$$
((f'\circ f)^*(T), {\cal G}(q(f'\circ f,T))(g))=
Sig({\cal G})(f'\circ f)(T,g)$$
where the first two equalities are by definition of $Sig({\cal G})$, the third by the composition property of $\cal G$, the fourth by the axioms of the C-system structure and the fifth again by the definition of $Sig({\cal G})$. This completes the proof of Lemma \ref{2016.08.28.l1}.
\end{proof}

One defines $Sig$ on morphisms of presheaves $r:{\cal G}\sr {\cal G}'$ by the family of morphisms
\begin{eq}
\llabel{2016.08.30.eq9}
Sig(r)_{\Gamma}(T,g)=(T,r_{T}(g))
\end{eq}
For $r:{\cal G}\sr {\cal G}'$ and $f:\Gamma'\sr \Gamma$, we have
$$Sig({\cal G})(f)\circ Sig(r)_{\Gamma'}=Sig(r)_{\Gamma}\circ Sig({\cal G}')(f)$$
that is, the family of functions $Sig(r)_{\Gamma}$ parametrized by $\Gamma\in CC$ is a morphism of presheaves.

For ${\cal G}\in PreShv(CC)$ we have 
\begin{eq}
\llabel{2016.12.14.eq1}
Sig(Id_{\cal G})_{\Gamma}(T,g)=(T,(Id_{\cal G})_{T}(g))=(T,g)
\end{eq}
and for $r:{\cal G}\sr {\cal G}'$, $r':{\cal G}'\sr {\cal G}''$ we have
\begin{eq}
\llabel{2016.12.14.eq2}
Sig(r\circ r')_{\Gamma}(T,g)=(T,(r\circ r')_{T}(g))=(T,r'_{T}(r_{T}(g)))=Sig(r')(Sig(r)(T,g))
\end{eq}
These two equalities show that the functor data given by $Sig$ on presheaves and $Sig$ on morphisms of presheaves is a functor that we also denote by 
$$Sig:PreShv(CC)\sr PreShv(CC)$$
\begin{remark}\rm
\llabel{2016.12.14.rem1}
The construction of $Sig$ works in more general setting than presheaves.

Indeed, for any family of sets $G(\Gamma)$ parametrized by $\Gamma\in CC$ the formula (\ref{2016.08.30.eq7}) defines a new family of sets $Sig(G)(\Gamma)$ also parametrized by $\Gamma\in CC$. For any two families $G,G'$ and a family of functions $r_{\Gamma}:G({\Gamma})\sr G'({\Gamma})$ the formula (\ref{2016.08.30.eq9}) defines a family of functions $Sig(r)_{\Gamma}:Sig(G)(X)\sr Sig(G')(X)$. The  properties (\ref{2016.12.14.eq1}) and (\ref{2016.12.14.eq2}) hold in this more general setting.

We can also define $Sig(G)$ for any presheaf data, that is, for any pair consisting of a family $G(\Gamma)$ of sets parametrized by $\Gamma\in CC$ and a family of functions $G(f):G(\Gamma)\sr G(\Gamma')$ parametrized by $f:\Gamma'\sr \Gamma$ in $Mor(CC)$. For this we can  again use formulas  (\ref{2016.08.30.eq7}) and (\ref{2016.08.30.eq8}). 

If $r_{\Gamma}:G({\Gamma})\sr G'({\Gamma})$ is a morphism of functor data, that is functions $r_{*}$ commute with functions $G(*)$, then $Sig(r)$ is a morphism of functor data as well.
\end{remark}
\begin{problem}
\llabel{2016.08.30.prob1}
For $n\ge 0$ to construct an isomorphism of presheaves
$$S\Ob_n:Sig(\Ob_n)\sr \Ob_{n+1}$$
\end{problem}
In constructing a solution of this problem and other problems where one needs to build an of isomorphism of presheaves we will use the following lemma that is often applied without an explicit reference.
\begin{lemma}
\llabel{2016.11.14.l1}
Let $\Phi,\Phi':{\cal C}\sr {\cal D}$ be functors and $\phi:\Phi\sr \Phi'$ a natural transformation. Then $\phi$ is an isomorphism of functors if and only if for all objects $X$ of $\cal C$ the morphism $\phi_X:\Phi(X)\sr \Phi'(X)$ is an isomorphism in $\cal D$.

When this condition is satisfied, the inverse isomorphism is formed by the family of morphisms $\phi^{-1}_{X}=(\phi_X)^{-1}$.
\end{lemma}
\begin{proof}
One should first verify that identity isomorphism of functors is formed by the family $Id_{\Phi(X)}$. This is immediate from the definitions.

If $\phi$ is an isomorphism as above and $\phi^{-1}$ is its inverse, then the composition of these two morphisms is $Id_{\Phi(X)}$ and since it consists of the identity morphisms of the objects $\Phi(X)$ we condclude that the morphisms $\phi^{-1}_X$ form inverses to the morphisms $\phi_X$. This proves the ``only if part''. 

If all morphisms $\phi_X$ are isomorphisms then the family $(\phi_X)^{-1}$ forms a morphism of presheaves $\phi^{-1}:\Phi'\sr \Phi$. Indeed, for $f:X\sr Y$ one has
$$\phi^{-1}_X\circ \Phi(f)=\Phi'(f)\circ \phi^{-1}_Y$$
This equality follows by taking its composition with $\phi_X$ on the left and $\phi_Y$ on the right. That $\phi^{-1}$ is both the left and the right inverse to $\phi$ is immediate from its definition. This proves the ``if'' part. 
\end{proof}
We will also need the following lemma.  
\begin{lemma}
\llabel{2016.09.01.l1}
Let $\Gamma\in CC$. Then one has:
\begin{enumerate}
\item if $T\in\Ob_1(\Gamma)$ and $X\in\Ob_n(T)$ then $X\in\Ob_{n+1}(\Gamma)$,
\item if $X\in \Ob_{n+1}(\Gamma)$ then $ft^n(X)\in \Ob_1(\Gamma)$ and $X\in \Ob_n(ft^n(X))$.
\end{enumerate}
\end{lemma}
\begin{proof}
The first assertion follows from the equalities $l(X)=l(T)+n=l(\Gamma)+1+n$ and $ft^{n+1}(X)=ft(ft^n(X))=ft(T)=\Gamma$. 

To prove the second assertion let $X\in\Ob_{n+1}(\Gamma)$. Since $l(X)\ge n$ we have $l(ft^n(X))=l(X)-n=l(\Gamma)+(n+1)-n=l(\Gamma)+1$. The equality $ft^1(ft^n(X))=ft^{n+1}(X)=\Gamma$ is obvious and we conclude that $ft^n(X)\in\Ob_1(\Gamma)$. Next, again because $l(X)\ge n$, we have $l(X)=l(ft^n(X))+n$ and since $ft^n(X)=ft^n(X)$ we have that  $X\in\Ob_n(ft^n(X))$. 
\end{proof}
\begin{construction}\rm
\llabel{2016.08.30.constr1}
Let $\Gamma\in CC$. Then $Sig(\Ob_n)(\Gamma)$ is the set of pairs $(T,X)$ where $T\in\Ob_1(\Gamma)$ and $X\in \Ob_n(T)$. By Lemma \ref{2016.09.01.l1}(1), the formula 
\begin{eq}
\llabel{2016.09.01.eq4}
S\Ob_{n,\Gamma}(T,X)=X
\end{eq}
defines a function $Sig(\Ob_n)(\Gamma)\sr \Ob_{n+1}(\Gamma)$.

Conversely, by Lemma \ref{2016.09.01.l1}(2), the formula 
\begin{eq}
\llabel{2016.09.01.eq5}
S\Ob^{-1}_{n,\Gamma}(X)=(ft^n(X),X)
\end{eq}
defines a function $\Ob_{n+1}(\Gamma)\sr Sig(\Ob_n)(\Gamma)$.

If $\Phi=S\Ob_{n,\Gamma}$ and $\Psi=S\Ob^{-1}_{n,\Gamma}$ then 
$$\Phi(\Psi(X))=\Phi((ft^n(X),X))=X$$
and
$$\Psi(\Phi(T,X))=\Psi(X)=(ft^n(X),X)=(T,X)$$
where the last equality follows from the equality $T=ft^n(X)$. We conclude that $S\Ob_{n,\Gamma}$ and $S\Ob^{-1}_{n,\Gamma}$ are mutually inverse bijections.

In view of Lemma \ref{2016.11.14.l1}, it remains to verify that the family of bijections $S\Ob_{n,\Gamma}$ parametrized by $\Gamma\in CC$ is a morphism of presheaves, that is, that for any $f:\Gamma'\sr\Gamma$ and $(T,X)\in Sig(\Ob_n)(\Gamma)$ we have
\begin{eq}
\llabel{2016.08.30.eq10}
\Ob_{n+1}(f)(S\Ob_{n,\Gamma}((T,X)))=S\Ob_{n,\Gamma'}(Sig(\Ob_n)(f)((T,X)))
\end{eq}
Computing we get
$$\Ob_{n+1}(f)(S\Ob_{n,\Gamma}((T,X)))=f^*(X)$$
$$S\Ob_{n,\Gamma'}(Sig(\Ob_n)(f)((T,X)))=S\Ob_{n,\Gamma'}(f^*(T),q(f,T)^*(X))=q(f,T)^*(X)$$
and (\ref{2016.08.30.eq10}) follows from \cite[Lemma 2.7]{fromunivwithPiI}. This completes Construction \ref{2016.08.30.constr1}.
\end{construction}
As a corollary of Construction \ref{2016.08.30.constr1} and Lemma \ref{2016.11.14.l1} we obtain the fact that the family of functions (\ref{2016.09.01.eq5})  parametrized by $\Gamma\in CC$ is an isomorphism of presheaves that is inverse to $S\Ob_n$.

We proceed now to the construction of isomorphisms $S\wt{Ob}_n$. 

Recall from \cite[Sec. 3]{Csubsystems}, that $\wt{Ob}(CC)$ is the set of elements $o\in Mor(CC)$ such that $o\in sec(p_{codom(o)})$ and $l(codom(o))>0$. For such elements we also denote $codom(o)$ by $\partial(o)$. 

It follows easily from (\ref{2016.11.15.eq6}) that for $\Gamma\in Ob(CC)$ and $n>0$ one has $o\in\wt{\Ob}_n(\Gamma)$ if and only if $o\in\wt{Ob}(CC)$ and $\partial(o)\in \Ob_n(\Gamma)$. It also follows from (\ref{2016.11.15.eq6}) that $\Ob_0(\Gamma)=\emptyset$. 
\begin{problem}
\llabel{2016.08.30.prob2}
For $n\ge 1$ to construct an isomorphism of presheaves
$$S\wt{\Ob}_n:Sig(\wt{\Ob}_n)\sr \wt{\Ob}_{n+1}$$
\end{problem}
\begin{lemma}
\llabel{2016.11.18.l1}
Let $\Gamma\in CC$. Then one has:
\begin{enumerate}
\item if $T\in\Ob_1(\Gamma)$ and $o\in\wt{\Ob}_n(T)$ then $o\in\wt{\Ob}_{n+1}(\Gamma)$,
\item if $o\in \wt{\Ob}_{n+1}(\Gamma)$ then $ft^n(\partial(o))\in \Ob_1(\Gamma)$ and $o\in \wt{\Ob}_n(ft^n(\partial(o)))$.
\end{enumerate}
\end{lemma}
\begin{proof}
If $o\in\wt{\Ob}_n(T)$ we have $n>0$ an therefore $o\in\wt{Ob}(CC)$ and $\partial(o)\in \Ob_n(T)$. By Lemma \ref{2016.09.01.l1}(1) we have $\partial(o)\in \Ob_{n+1}(\Gamma)$. Therefore $o\in\wt{\Ob}_{n+1}(T)$. This proves the first assertion. 

If $o\in \wt{\Ob}_{n+1}(\Gamma)$ then $o\in\wt{Ob}(CC)$ and $\partial(o)\in \Ob_{n+1}(\Gamma)$. By Lemma \ref{2016.09.01.l1}(2) we have $ft^n(\partial(o))\in \Ob_1(\Gamma)$ and $\partial(o)\in \Ob_n(ft^n(\partial(o)))$.  Therefore $o\in\wt{\Ob}_n(ft^n(\partial(o)))$. 
\end{proof}
We can now provide a construction for Problem \ref{2016.08.30.prob2}.
\begin{construction}\rm\llabel{2016.09.01.constr2}
For $\Gamma\in CC$ we have 
$$Sig(\wt{\Ob}_n)(\Gamma)=\{(T,o)\,|\,T\in \Ob_1(\Gamma),\,\,o\in \wt{\Ob}_n(T)\}$$
For $(T,o)\in Sig(\wt{\Ob}_n)(\Gamma)$ we have $o\in\wt{\Ob}_{n+1}(\Gamma)$ by Lemma \ref{2016.11.18.l1}(1) and therefore the formula
\begin{eq}
\llabel{2016.09.01.eq6}
S\wt{\Ob}_{n,\Gamma}(T,o)=o
\end{eq}
defines a function $Sig(\wt{\Ob}_n)(\Gamma)\sr \wt{\Ob}_{n+1}(\Gamma)$.

If $o\in \wt{\Ob}_{n+1}(\Gamma)$ then by Lemma \ref{2016.11.18.l1}(2), $ft^n(\partial(o))\in \wt{Ob}_1(\Gamma)$ and $o\in \wt{Ob}_n(ft^n(\partial(o)))$. Therefore the formula
\begin{eq}
\llabel{2016.09.01.eq7}
S\wt{Ob}^{-1}_{n,\Gamma}(o)=(ft^n(\partial(o)),o)
\end{eq}
defines a function $\wt{\Ob}_{n+1}(\Gamma)\sr Sig(\wt{\Ob}_n)(\Gamma)$.

One verifies in the same way as in Construction \ref{2016.08.30.constr1} that $S\wt{\Ob}_{n,\Gamma}$ and $S\wt{Ob}^{-1}_{n,\Gamma}$ are mutually inverse bijections.

In view of Lemma \ref{2016.11.14.l1} it remains to verify that the family of functions  $S\wt{\Ob}_{n,\Gamma}$ parametrized by $\Gamma\in CC$ is a morphism of functors, that is, that for $f:\Gamma'\sr\Gamma$ and $(T,o)\in S\Ob_{n,\Gamma}$ one has
\begin{eq}
\llabel{2016.09.01.eq2b}
\wt{\Ob}_{n+1}(f)(S\wt{\Ob}_{n,\Gamma}(T,o))=S\wt{\Ob}_{n,\Gamma'}(Sig(\wt{\Ob}_n)(f)(T,o))
\end{eq}
Computing we get
$$\wt{\Ob}_{n+1}(f)(S\wt{\Ob}_{n,\Gamma}(T,o))=\wt{\Ob}_{n+1}(f)(o)=f^*(o)$$
$$S\wt{\Ob}_{n,\Gamma'}(Sig(\wt{\Ob}_n)(f)(T,o))=S\wt{\Ob}_{n,\Gamma'}(f^*(T),q(f,T)^*(o))=q(f,T)^*(o)$$
and we conclude that (\ref{2016.09.01.eq2b}) holds by \cite[Lemma 2.15]{fromunivwithPiI}.

This completes Construction \ref{2016.09.01.constr2}.
\end{construction}
As a corollary of Construction \ref{2016.09.01.constr2} and Lemma \ref{2016.11.14.l1} we obtain the fact that the family of functions (\ref{2016.09.01.eq7})  parametrized by $\Gamma\in CC$ is an isomorphism of presheaves that is inverse to $S\wt{\Ob}_n$.
\begin{lemma}
\llabel{2016.12.04.l1}
For any $n\ge 1$ the square of morphisms of presheaves
\begin{eq}
\llabel{2016.12.04.eq1}
\begin{CD}
Sig(\wt{\Ob}_n) @>S\wt{\Ob}_n>> \wt{\Ob}_{n+1}\\
@VSig(\partial) VV @VV\partial V\\
Sig(\Ob_n) @>S\Ob_n>> \Ob_{n+1}
\end{CD}
\end{eq}
commutes.
\end{lemma}
\begin{proof}
Let $\Gamma\in CC$. By definition we have
$$Sig(\wOb_n)(\Gamma)=\{(T,o)\,|\,T\in \Ob_1(\Gamma), o\in \wOb_n(T)\}$$
Let $(T,o)\in Sig(\wOb_n)(\Gamma)$. Then, again by definitions,
$$\partial_{\Gamma}(S\wOb_{n,\Gamma}(T,o))=\partial_{\Gamma}(o)$$
and
$$S\Ob_{n,\Gamma}(Sig(\partial)_{\Gamma}(T,o))=S\Ob_{n,\Gamma}(T,\partial_{\Gamma}(o))=\partial_{\Gamma}(o)$$
The lemma is proved. 
\end{proof}
\begin{remark}\rm
\llabel{2016.11.18.rem1}
Define $Sig^n$ by induction on $n$, setting $Sig^0=Id_{PreShv(CC)}$ and $Sig^{n+1}=Sig^n\circ Sig$. Then, also by induction on $n$,  we can construct isomorphisms 
$$S\Ob^n_m:Sig^n(\Ob_m)\sr \Ob_{n+m}$$
where $S\Ob^0_m=Id_{\Ob_m}$ and $S\Ob^{n+1}_m$ is the composition 
$$
\begin{CD}
Sig^{n+1}(\Ob_m)= Sig(Sig^n(\Ob_m)) @>Sig(S\Ob^n_m)>> Sig(\Ob_{n+m}) @>S\Ob_{n+m}>> \Ob_{n+m+1}
\end{CD}
$$
In exactly the same way we construct isomorphisms
$$S\wt{\Ob}^n_m:Sig^n(\wt{\Ob}_m)\sr \wt{\Ob}_{n+m}$$
\end{remark}

\subsection{Functor $D_p$, sets $D_p^n(X,Y)$ and the $\circ$-notation}
\llabel{Sec.2}

In this section we work in the context of a category $\cal C$ with a universe $p$. The goal of the section is to construct, for any such pair, a functor
$$D_p:PreShv({\cal C})\sr PreShv({\cal C})$$
The definition of a universe in a category was given in \cite[Definition 2.1]{Cfromauniverse}. We repeat it here for the convenience of the reader.
\begin{definition}
\llabel{2009.11.1.def1}
Let $\cal C$ be a category. A universe structure on a morphism $p:\wt{U}\sr U$ in $\cal C$ is a mapping that assigns to any morphism $f:X\sr U$ in $\cal C$ a pullback of the form
\begin{eq}
\llabel{2016.12.02.eq8}
\begin{CD}
(X,f) @>Q(F)>>  \wt{U}\\
@Vp_{X,F}VV  @VVpV\\
X @>F>> U
\end{CD}
\end{eq}
A universe in $\cal C$ is a morphism together with a universe structure on it. 
\end{definition}
We usually refer to a universe by the name of the corresponding morphism without mentioning the choices of pullbacks explicitly. To shorten the notation we will write $p_F$ instead of $p_{X,F}$. 

For $f:W\sr X$ and $g:W\sr \wt{U}$ such that $f\circ F=g\circ p$ we will denote by $f*_Fg$ the unique morphism  $W\sr (X;F)$ such that 
\begin{eq}
\llabel{2016.11.10.eq1a}
(f*_Fg)\circ p_{F}=f
\end{eq}
\begin{eq}
\llabel{2016.11.10.eq1b}
(f*_Fg)\circ Q(F)=g
\end{eq}
For $X'\stackrel{f}{\sr}X\stackrel{F}{\sr}U$ we let $Q(f,F)$ denote the morphism
\begin{eq}
\llabel{2016.12.02.eq4}
Q(f,F)=(p_{f\circ F}\circ f)*_FQ(f\circ F):(X';f\circ F)\sr (X;F)
\end{eq}
Observe that one has
\begin{eq}
\llabel{2016.08.24.eq4}
Q(f\circ F)=Q(f,F)\circ Q(F)
\end{eq}
\begin{eq}
\llabel{2016.08.26.eq2}
Q(Id_X,F)=Id_{(X;F)}
\end{eq}
\begin{eq}
\llabel{2016.08.26.eq3}
Q(f'\circ f,F)=Q(f',f\circ F)\circ Q(f,F)
\end{eq}
where the first equality follows directly from the definition, the second from the definition and the uniqueness of the morphisms $f*_Fg$ satisfying (\ref{2016.11.10.eq1a}) and (\ref{2016.11.10.eq1b}) and the third is proved in \cite[Lemma 2.5]{Cfromauniverse}.

Let us fix a category $\cal C$ and a universe $p$ in it. 

For any ${\cal G}\in PreShv({\cal C})$ we define functor data $D_p({\cal G})$ given on objects by
\begin{eq}
\llabel{2016.08.30.eq4}
D_p({\cal G})(X) := \amalg_{F:X\sr U}{\cal G}((X;F))
\end{eq}
and on morphisms by
\begin{eq}
\llabel{2016.08.30.eq5}
D_p({\cal G})(f):(F,\gamma)\mapsto (f\circ F, {\cal G}(Q(f,F))(\gamma))
\end{eq}
\begin{lemma}
\llabel{2016.09.07.l1}
The functor data $D_p({\cal G})$ specified above is a presheaf, i.e., one has
\begin{enumerate}
\item for any $X\in {\cal C}$, $D_p({\cal G})(Id_X)=Id_{D_p({\cal G})(X)}$,
\item for any $f:X\sr Y$, $g:Y\sr Z$ in $\cal C$,
$$D_p({\cal G})(f\circ g)=D_p({\cal G})(g)\circ D_p({\cal G})(f)$$
\end{enumerate}
\end{lemma}
\begin{proof}
For the first property we have
$$D_p({\cal G})(Id_X)((F,\gamma))=(Id_X\circ F, {\cal G}(Q(Id_X,F))(\gamma))=(F, \gamma)$$
where the second equality is by (\ref{2016.08.26.eq2}) and the identity morphism axiom for  the presheaf ${\cal G}$.

For the second one we have
$$D_p({\cal G})(f\circ g)(F,\gamma)=(f\circ g\circ F,{\cal G}(Q(f\circ g,F))(\gamma)))=$$$$
(f\circ g\circ F, {\cal G}(Q(f,g\circ F)\circ Q(g,F))(\gamma))=
(f\circ (g\circ F), {\cal G}(Q(f,g\circ F))({\cal G}(Q(g,F))(\gamma)))=$$
$$D_p({\cal G})(f)(D_p({\cal G})(g)(F,\gamma))=(D_p({\cal G})(g)\circ D_p({\cal G})(f))(F,\gamma)$$
where the second equality is by (\ref{2016.08.26.eq3}) and the third one by the composition axiom of the presheaf $\cal G$.
\end{proof}

One defines $D_p$ on morphisms of presheaves $r:{\cal G}\sr {\cal G}'$ by the family of morphisms 
\begin{eq}
\llabel{2016.08.30.eq6}
D_p(r)_X(F,\gamma)=(F,r_{(X;F)}(\gamma))
\end{eq}
For $f:X\sr X'$ and $r:{\cal G}\sr {\cal G}'$ we have
\begin{eq}
\llabel{2016.11.14.eq2}
D_p({\cal G})(f)\circ D_p(r)_X=D_p(r)_{X'}\circ D_p({\cal G}')(f)
\end{eq}
that is, the family of functions $D_p(r)_X$ parametrized by $X\in {\cal C}$ is a morphism of presheaves. 

For ${\cal G}\in PreShv({\cal C})$ we have 
%
$$D_p(Id_{{\cal G}})_X=Id_{D_p({\cal G})(X)}$$
and for $r:{\cal G}\sr {\cal G}'$ and $r':{\cal G}'\sr {\cal G}''$ we have 
\begin{eq}
\llabel{2016.12.18.eq4}
D_p(r\circ r')_X=D_p(r)_X\circ D_p(r')_X
\end{eq}
These two equalities show that the functor data given by $D_p$ on presheaves and $D_p$ on morphisms of presheaves is a functor that we also denote by
$$D_p:PreShv({\cal C})\sr PreShv({\cal C})$$
Note that for the presheaves of the form $Yo(A)$, where $Yo$ is the Yoneda embedding, we have
\begin{eq}
\llabel{2016.11.14.eq4}
D_p(Yo(A))(X)=\amalg_{F:X\sr U}Mor_{\cal C}((X;F),A)
\end{eq} 
and for a morphism $f:X\sr X'$,
\begin{eq}
\llabel{2016.11.14.eq4a}
D_p(Yo(A))(f)(F_1,F_2)=(f\circ F_1, Q(f,F_1)\circ F_2)
\end{eq}
For a morphism $a:A'\sr A$ we have
\begin{eq}
\llabel{2016.12.02.eq6}
D_p(Yo(a))_X(F_1,F_2)=(F_1,F_2\circ a)
\end{eq}
Define 
\begin{eq}
\llabel{2016.12.24.eq1}
D_p^n(X,Y)=D_p^{n}(Yo(Y))(X)
\end{eq}
such that in particular one has
$$D_p^0(X,Y)=Mor_{\cal C}(X,Y)$$
Since $D_p^n(Yo(Y))$ is a presheaf we have, for any $f:X'\sr X$, the function 
$$D_p^n(Yo(Y))(f):D_p^n(Yo(Y))(X)\sr D_p^n(Yo(Y))(X')$$
that we denote by 
\begin{eq}
\llabel{2016.12.24.eq2}
D_p^n(f,Y):D_p^n(X,Y)\sr D_p^n(X',Y)
\end{eq}
Since $D_p^n$ and $Yo$ are functors we have, for any $g:Y\sr Y'$, a function
$$D_p^n(Yo(g))_X:D_p^n(Yo(Y))(X)\sr D_p^n(Yo(Y'))(X)$$
that we denote by
\begin{eq}
\llabel{2016.12.24.eq3}
D_p^n(X,g):D_p^n(X,Y)\sr D_p^n(X,Y')
\end{eq}
Let $d\in D_p^n(X,Y)$. For $f:X'\sr X$ we let $f\circ^n d$ denote $D_p^n(f,Y)(d)$. For $g:Y\sr Y'$ we let $d \,^n\!\circ g$ denote $D_p^n(X,g)(d)$. When no confusion is possible we will abbreviate both $\circ^n$ and $^n\circ$ to $\circ$.

Let us summarize, using this ``$\circ$-notation'' some of the results proved above in the following lemma.
\begin{lemma}
\llabel{2017.01.07.l1}
For $d\in D_p^n(X,Y)$ we have the following formulas:
\begin{enumerate}
\item $Id_X\circ d=d$,
\item $(f'\circ f)\circ d=f'\circ (f\circ d)$,
\item $d\circ Id_Y=d$,
\item $d\circ (g\circ g')=(d\circ g)\circ g'$,
\item $f\circ (d\circ g)=(f\circ d)\circ g$.
\end{enumerate}
\end{lemma}
\begin{proof}
The first two equalities follow from the axioms of presheaf for $D_p^n(Yo(Y))$, the second two from the fact that $Yo\circ D_p^n$ is a functor and the last one from the fact that the family of functions $D_p(Yo(g))_{-}$ is a morphism of presheaves. 
\end{proof}
\begin{lemma}
\llabel{2016.12.24.l1}
Let $({\cal C},p)$ be a universe category, $n\ge 1$, $X,Y\in {\cal C}$, and 
$$(F,a)\in \amalg_{F:X\sr U}D_p^{n-1}((X;F),Y)=D_p(D_p^{n-1}(Yo(Y)))(X)=D_p^n(X,Y)$$
Then one has
\begin{enumerate}
\item for $f:X'\sr X$ 
$$f\circ (F,a)=(f\circ F,Q(f,F)\circ a)$$
\item for $g:Y\sr Y'$ 
$$(F,a)\circ g=(F,a\circ g)$$
\end{enumerate}
\end{lemma}
\begin{proof}
In the first case we have
$$f\circ (F,a)=
$$$$D_p^n(Yo(Y))(f)((F,a))=
(f\circ F,D_p^{n-1}(Yo(Y))(Q(f,F))(a))=
$$$$(f\circ F,Q(f,F)\circ a)$$
where the first equality is by the definition of $D_p^n(f,Y)$, the second by (\ref{2016.08.30.eq5}) and the third by the definition of $D_p^{n-1}(Q(f,F),Y)$. 

In the second case we have
$$(F,a)\circ g=
$$$$D_p^n(Yo(g))_X((F,a))=
(F,D_p^{n-1}(Yo(g))_{(X;F)}(a))=
$$$$(F,a\circ g)
$$
where the first equality is by the definition of $D_p^n(X,g)$, the second by (\ref{2016.08.30.eq5}) and the third by the definition of $D_p^{n-1}((X;F),g)$. The lemma is proved.
\end{proof}
\begin{remark}\rm
\llabel{2015.07.29.rem2}
It is likely that the functions (\ref{2016.12.24.eq2}) and (\ref{2016.12.24.eq3}) generalize to composition functions
\begin{eq}
\llabel{2016.12.18.eq3}
D_p^n(X,Y)\times D_p^m(Y,Z)\sr D_p^{n+m}(X,Z)
\end{eq}
The formulas 1.-5. suggest that these composition functions satisfy the unity and associativity axioms and therefore one obtains, from any universe category $({\cal C},p)$, a new category $({\cal C},p)_*$ with the same collection of objects and morphisms between two objects given by
$$Mor_{({\cal C},p)_*}(X,Y)=\amalg_{n\ge 0}D_p^n(X,Y)$$
The study of the composition functions (\ref{2016.12.18.eq3}) and categories $({\cal C},p)_*$ is deferred to a later paper.  
\end{remark}
\begin{remark}\rm
\llabel{2016.12.14.rem2}
The observations of Remark \ref{2016.12.14.rem1} apply, with obvious modifications, to the construction $D_p$ as well.
\end{remark}

\subsection{Isomorphisms of presheaves $u_1$ and $\wt{u}_1$}
\llabel{Sec.3}

We now consider a universe category, that is, a category $\cal C$ with a universe $p$ and a choice of a final object $pt$. We usually denote a universe category as $({\cal C},p)$ without mentioning the final object. For any universe category we have constructed in \cite[Section 2]{Cfromauniverse} a C-system $CC({\cal C},p)$.

The main goal of this section is to provide constructions for Problems 
\ref{2015.04.30.prob1a} and \ref{2015.04.30.prob1b}.

Let us first recall the construction of $CC({\cal C},p)$. One defines first, by induction on $n$, pairs $(Ob_n, int_n:Ob_n\sr {\cal C})$ where $Ob_n=Ob_n({\cal C},p)$ is a set and $int_n$ is a function from $Ob_n$ to objects of $\cal C$. The definition is as follows:
\begin{enumerate}
\item $Ob_0$ is the standard one point set $unit$ whose element we denote by $tt$. The function $int_0$ maps $tt$ to the final object $pt$ of the universe category structure on $\cal C$,
\item $Ob_{n+1}=\amalg_{A\in Ob_n}Mor(int(A),U)$ and $int_{n+1}(A,F)=(int(A);F)$.
\end{enumerate}
We then define $Ob(CC({\cal C},p))$ as $\amalg_{n\ge 0}Ob_n$ such that elements of $Ob(CC({\cal C},p))$ are pairs $\Gamma=(n,A)$ where $A\in Ob_n({\cal C},p)$. We define the function $int:Ob(CC({\cal C},p))\sr {\cal C}$ as the sum of functions $int_n$. Where no confusion between $int$ and $int_n$ is likely we will omit the index $n$ at $int_n$. 

The morphisms in $CC({\cal C},p)$ are defined by
$$Mor_{CC({\cal C},p)}=\amalg_{\Gamma,\Gamma'\in Ob(CC)}Mor_{\cal C}(int(\Gamma),int(\Gamma'))$$
and the function $int$ on morphisms maps a triple $((\Gamma,\Gamma'),a)$ to $a$.
Note that the subset in $Mor$ that consists of $f$ such that $dom(f)=\Gamma$ and $codom(f)=\Gamma'$ is not equal to the set $Mor_{\cal C}(int(\Gamma),int(\Gamma'))$ but instead to the set of triples of the form $f=((\Gamma,\Gamma'),a)$ where $a\in Mor_{\cal C}(int(\Gamma),int(\Gamma'))$. The functor $int$ maps $((\Gamma,\Gamma'),a)$ to $a$. This map is bijective and therefore the functor is fully faithful but its morphism component is not the identity function. 

The length function is defined by $l((n,A))=n$.

One defines $pt$ as $pt=(0,tt)$. It is the only object of length $0$.

If $\Gamma=(n,B)$ where $n>0$ then, by construction, $B=(A,F)$ where $F:int(A)\sr U$. The $ft$ function is defined on such $\Gamma$ by $ft(\Gamma)=(n-1,A)$ and on $pt$ by $ft(pt)=pt$.
\begin{lemma}
\llabel{2016.08.22.l1}
For $\Gamma=(n,A)$ and  $T=(n',B)\in Ob(CC({\cal C},p))$ one has $T\in\Ob_1(\Gamma)$ if and only if $n'=n+1$ and there exists $F:int(A)\sr U$ such that $B=(A,F)$.
\end{lemma}
\begin{proof}
By definition of the length function $l$, we have $l(\Gamma)=n$ and $l(T)=n'$. By definition of $\Ob_1$, $T\in \Ob_1(\Gamma)$ if and only if $n'=n+1$ and $ft(T)=\Gamma$.

If $T=(n+1,(A,F))$ then $n'=n+1$. In particular, $l(T)>0$ and therefore $ft(T)=(n,A)=\Gamma$. This proves the "if" part. 

Assume that $T=(n',B)\in \Ob_1(\Gamma)$. Then $n'=n+1$. Since $n'>0$, $B$ is a pair of the form $(A',F)$. Since $ft(T)=(n,A')=(n,A)$ we have $A'=A$. This proves the ``only if'' part. 
\end{proof}
The $p$-morphism for $\Gamma=(n,A)$ where $n>0$ and $A=(B,F)$ is given by $((\Gamma,ft(\Gamma)),p_{F})$ where $p_{F}$ are the $p$-morphisms of the universe structure.

For $f:(n,A')\sr (n,A)$ and $T$ such that $l(T)=l(\Gamma)+1$ and $ft(T)=\Gamma$ one has, by Lemma \ref{2016.08.22.l1}, $T=(n+1,(A,F))$ and one defines 
\begin{eq}
\llabel{2016.08.22.eq2}
f^*(T)=(n+1,(A',int(f)\circ F))
\end{eq}
and
\begin{eq}
\llabel{2016.08.22.eq3}
q(f,T)=((f^*(T),T),Q(int(f),F))
\end{eq}
The C-system axioms are verified in \cite{Cfromauniverse}.

Let us denote by
$$int^{\circ}:PreShv({\cal C})\sr PreShv(CC)$$
the functor of pre-composition with $int^{op}$ and by 
$$Yo:{\cal C}\sr PreShv({\cal C})$$
the Yoneda embedding of $\cal C$. 

\begin{problem}
\llabel{2015.04.30.prob1a}
To construct an isomorphism of presheaves 
\begin{eq}
\llabel{2016.11.12.eq2}
u_1:\Ob_1\sr int^{\circ}(Yo(U))
\end{eq}
such that for $\Gamma=(n,A)$ and $T=(n+1,(A,F))$ one has 
\begin{eq}
\llabel{2015.04.30.eq3a}
u_{1,\Gamma}(T)=F
\end{eq}
\end{problem}
\begin{construction}\rm\llabel{2016.08.22.constr1}
By definition of $int^{\circ}$ and $Yo$ and Lemma \ref{2016.11.14.l1}, an isomorphism of presheaves of the form (\ref{2016.11.12.eq2}) is a family of functions of the form
$$u_{1,\Gamma}:\Ob_1(\Gamma)\sr Mor_{\cal C}(int(\Gamma),U)$$
parametrized by $\Gamma\in Ob(CC({\cal C},p))$ such that for any $f:\Gamma'\sr \Gamma$ and any $T\in \Ob_1(\Gamma)$ one has
\begin{eq}
\llabel{2015.04.30.eq1a}
u_{1,\Gamma'}(f^*(T))=int(f)\circ u_{1,\Gamma}(T)
\end{eq}
and for any $\Gamma$ the function $u_{1,\Gamma}$ is a bijection.

By Lemma \ref{2016.08.22.l1}, the conditions (\ref{2015.04.30.eq3a}) define our family completely and it remains to verify (\ref{2015.04.30.eq1a}) and the bijectivity condition. 

For $\Gamma=(n,A)$, $T=(n+1,(A,F))$, $\Gamma'=(n',A')$ and $f:\Gamma'\sr \Gamma$ we have, by (\ref{2016.08.22.eq2}),
$$f^*(T)=(n'+1,(A', int(f)\circ F))$$
Therefore,
$$u_{1,\Gamma'}(f^*(T))=u_{1,\Gamma'}((n'+1,(A', int(f)\circ F)))=int(f)\circ F=int(f)\circ u_{1,\Gamma}(T)$$
which proves (\ref{2015.04.30.eq1a}).

By Lemma \ref{2016.08.22.l1}, for $\Gamma=(n,A)$, the formula $F\mapsto (n+1,(A,F))$ defines a function 
$$Mor_{\cal C}(int(A),U)\sr \Ob_1(\Gamma)$$
By the same lemma and (\ref{2015.04.30.eq3a}) this function is inverse to $u_{1,\Gamma}$. This proves the bijectivity condition and completes Construction \ref{2016.08.22.constr1}.
\end{construction}
Using again Lemma \ref{2016.08.22.l1} and (\ref{2015.04.30.eq3a}) we see that for any $\Gamma\in Ob(CC({\cal C},p))$ and $T\in \Ob_1(\Gamma)$,
\begin{eq}
\llabel{2015.05.02.eq1a}
int(T)=(int(\Gamma); u_{1,\Gamma}(T))
\end{eq}
and
\begin{eq}
\llabel{2016.08.24.eq3}
int(p_T)=p_{u_{1,\Gamma}(T)}
\end{eq}
For $f:\Gamma'\sr \Gamma$ and $T$ as above we have
\begin{eq}
\llabel{2016.08.30.eq3}
int(q(f,T))=Q(int(f),u_{1,\Gamma}(T))
\end{eq}
\begin{lemma}
\llabel{2016.08.22.l2}
For $\Gamma=(n,A)$ and $o\in \wt{\Ob_1}(\Gamma)$ one has 
\begin{eq}
\llabel{2016.08.22.eq1}
codom(int(o))=(int(\Gamma);u_{1,\Gamma}(\partial(o)))
\end{eq}
\end{lemma}
\begin{proof}
We have $codom(o)=\partial(o)\in \Ob_1(\Gamma)$. Therefore (\ref{2016.08.22.eq1}) follows from the equality $codom(int(f))=int(codom(f))$ and (\ref{2015.05.02.eq1a}).
\end{proof}
The second problem whose solution is constructed in this section is as follows.
\begin{problem}
\llabel{2015.04.30.prob1b}
To construct an isomorphism of presheaves 
\begin{eq}
\llabel{2016.11.12.eq3}
\wt{u}_1:\wt{\Ob}_1\sr int^{\circ}(Yo(\wt{U}))
\end{eq}
such that for $o\in \wOb_1(\Gamma)$ one has 
\begin{eq}
\llabel{2015.04.30.eq4a}
\wt{u}_{1,\Gamma}(o)=int(o)\circ Q(u_{1,\Gamma}(\partial(o)))
\end{eq}
where the right hand side is defined by (\ref{2016.08.22.eq1}) and the equality $dom(Q(F))=(dom(F);F)$.
\end{problem}
To construct a solution for this problem we will need the following two lemmas.
\begin{lemma}
\llabel{2016.08.26.l1}
For a universe $p$ in $\cal C$ and $X\in {\cal C}$, the function
$$\amalg_{F\in Mor(X,U)}sec(p_{F})\sr Mor(X,\wt{U})$$
given by the formula $(F,s)\mapsto s\circ Q(F)$ is a bijection.  The inverse bijection is given by the formula $\wt{F}\mapsto (\wt{F}\circ p, Id_X*_{\wt{F}\circ p}\wt{F})$ where $Id_X*_{\wt{F}\circ p}\wt{F}$ is defined because $Id_X\circ \wt{F}\circ p=\wt{F}\circ p$.
\end{lemma}
\begin{proof}
Let us denote the first function by $\Phi$ and second one by $\Psi$. We have
$$\Phi(\Psi(\wt{F}))=\Phi(\wt{F}\circ p, Id_X*_{\wt{F}\circ p} \wt{F})=(Id_X*_{\wt{F}\circ p}\wt{F})\circ Q(\wt{F}\circ p)=\wt{F}$$
where the last equality is by the definition of $*_{\wt{F}\circ p}$, and
$$\Psi(\Phi(F,s))=\Psi(s\circ Q(F))=((s\circ Q(F))\circ p, Id_X*_{(s\circ Q(F))\circ p}(s\circ Q(F)))$$
Next we have 
\begin{eq}
\llabel{2016.11.12.eq1}
s\circ Q(F)\circ p=s\circ p_{F}\circ F=F
\end{eq}
It remains to compare $Id_X*_{s\circ Q(F)\circ p}(s\circ Q(F))$ with $s$. 
To do it we need to compare its post-compositions with $p_{F}$ and $Q(F)$ with the same post-compositions for $s$. 

By (\ref{2016.11.12.eq1}) we may replace ${s\circ Q(F)\circ p}$ with $F$.  
We have
$$Id_X*_{F}(s\circ Q(F))\circ p_{F}=Id_X=s\circ p_{F}$$
$$Id_X*_{F}(s\circ Q(F))\circ Q(F)=s\circ Q(F)=s\circ Q(F)$$
Therefore, $Id_X*_{F}(s\circ Q(F))=s$ and
$$\Psi(\Phi(F,s))=(F,s)$$
The lemma is proved.
\end{proof}
\begin{lemma}
\llabel{2016.08.26.l4}
Let $p:Y\sr X$ be a morphism in $\cal C$ and $\Phi:{\cal C}\sr {\cal C}'$ a functor. Then for $s\in sec(p)$ one has $\Phi(s)\in sec(\Phi(p))$. 

If $\Phi$ is fully faithful then the resulting function 
$$\Phi_{sec,p}:sec(p)\sr sec(\Phi(p))$$
is a bijection.
\end{lemma}
\begin{proof}
The first assertion follows immediately from the definition of $sec$ and the axioms of a functor. 

Assume that $\Phi$ is fully faithful. To prove that $\Phi_{sec,p}$ is a bijection let 
$$\Phi^{-1}_{A,B}:Mor_{{\cal C}'}(\Phi(A),\Phi(B))\sr Mor_{\cal C}(A,B)$$
be the inverse to the function $\Phi_{A,B}:Mor_{\cal C}(A,B)\sr Mor_{{\cal C}'}(\Phi(B),\Phi(B))$ that we denoted simply by $\Phi$. One verifies easily that for any $A,B,C\in {\cal C}$ the functions $\Phi^{-1}_{A,B},\Phi^{-1}_{B,C}$ and $\Phi^{-1}_{A,C}$ commute with the compositions and for any $A\in {\cal C}$ one has $\Phi^{-1}_{A,A}(Id_{\Phi(A)})=Id_A$. 

Therefore, for $s'\in sec(\Phi_{Y,X}(p))$ we have $\Phi^{-1}_{X,Y}(s')\in sec(p)$. Indeed, 
$$\Phi_{X,X}(\Phi^{-1}_{X,Y}(s')\circ p)=\Phi_{X,Y}(\Phi^{-1}_{X,Y}(s'))\circ \Phi_{Y,X}(p)=s'\circ \Phi_{Y,X}(p)=Id_{\Phi(X)}$$
and since $\Phi_{X,X}^{-1}(Id_{\Phi(X)})=Id_X$ we obtain that $\Phi^{-1}_{X,Y}(s')\circ p=Id_X$. This implies that $\Phi_{sec,p}$ and the restriction of $\Phi^{-1}_{X,Y}$ to $sec(\Phi(p))$ form a pair of functions between $sec(p)$ and $sec(\Phi(p))$ and one sees immediately that they are mutually inverse.
\end{proof}
\begin{construction}\rm\llabel{2016.08.22.constr2}
By definition of $int^{\circ}$ and $Yo$ and Lemma \ref{2016.11.14.l1}, an isomorphism of presheaves of the form (\ref{2016.11.12.eq3}) is a family of functions of the form
$$\wt{u}_{1,\Gamma}:\wt{\Ob}_1(\Gamma)\sr Mor_{\cal C}(int(\Gamma),\wt{U})$$
parametrized by $\Gamma\in Ob(CC({\cal C},p))$ such that for any $f:\Gamma'\sr \Gamma$ and any $o\in \wt{\Ob}_1(\Gamma)$ one has
\begin{eq}
\llabel{2015.04.30.eq1b}
\wt{u}_{1,\Gamma'}(f^*(o))=int(f)\circ \wt{u}_{1,\Gamma}(o)
\end{eq}
and for any $\Gamma$ the function $\wt{u}_{1,\Gamma}$ is a bijection.

The equalities (\ref{2015.04.30.eq4a}) define our family completely and it remains to prove (\ref{2015.04.30.eq1b}) and the bijectivity condition. 

For the proof of (\ref{2015.04.30.eq1b}) we have the following, where we write $u$ instead of $u_{1,\Gamma}$ and $u_{1,\Gamma'}$ and $\wt{u}$ instead of $\wt{u}_{1,\Gamma}$ and $\wt{u}_{1,\Gamma'}$,
$$\wt{u}(f^*(o))=
int(f^*(o))\circ Q(u(\partial(f^*(o))))=
int(f^*(o))\circ Q(u(f^*(\partial(o))))=$$$$
int(f^*(o))\circ Q(int(f)\circ u(\partial(o)))=
int(f^*(o))\circ Q(int(f), u(\partial(o)))\circ Q(u(\partial(o)))=$$$$
int(f^*(o))\circ int(q(f,\partial(o)))\circ Q(u(\partial(o)))=
int(f^*(o)\circ q(f,\partial(o)))\circ Q(u(\partial(o)))=$$$$
int(q(f,\Gamma)\circ o)\circ Q(u(\partial(o)))=
int(f\circ o)\circ Q(u(\partial(o)))=$$$$
int(f)\circ int(o)\circ Q(u(\partial(o)))=
int(f)\circ \wt{u}(o)$$
where the first equality is by (\ref{2015.04.30.eq4a}),
second is by definition of $f^*(o)$,
the third is by (\ref{2015.04.30.eq1a}),
the fourth is by (\ref{2016.08.24.eq4}), 
the fifth is by (\ref{2016.08.30.eq3}),
the sixth is because $int$ is a functor, 
the seventh is by \cite[(2.19)]{fromunivwithPiI},
the eighth is by definition of $q(f,-)$,
the ninth is because $int$ is a functor
and the tenth is again by (\ref{2015.04.30.eq4a}). This completes the proof of (\ref{2015.04.30.eq1b}).

To prove that the function $\wt{u}_{1,\Gamma}$ is a bijection we will represent it as the composition of functions that we can show to be bijections. The functions are of the form
$$\wt{\Ob}_1(\Gamma)\sr \amalg_{T\in\Ob_1(\Gamma)}\partial^{-1}(T)\sr \amalg_{F:int(\Gamma)\sr U}sec(p_F)\sr Mor(int(\Gamma),\wt{U})$$
and are given by the formulas
$$o\mapsto (\partial(o),o)\spc\spc (T,o)\mapsto (u(T),int(o))\spc\spc  (F,s)\mapsto s\circ Q(F)$$
The first function is the function $X\sr \amalg_{y\in Y}f^{-1}(y)$, which is defined and is a bijection for any function of sets $f:X\sr Y$. The second one is the total function of the function $u$ and the family of functions $int_{sec,p_T}$ of Lemma \ref{2016.08.26.l4}. Since $u$ and the functions $int_{sec,p_T}$ are bijections the total function is a bijection. The third function is the bijection of Lemma \ref{2016.08.26.l1}.

Let us show that the composition of these bijections equals $\wt{u}$. Indeed, for $o\in \wt{\Ob}_1(\Gamma)$ we have
$$o\mapsto (\partial(o),o)\mapsto (u(\partial(o)),int(o))\mapsto int(o)\circ Q(u(\partial(o)))=\wt{u}(o)$$
This completes Construction \ref{2016.08.22.constr2}.
\end{construction}
\begin{remark}\rm\llabel{2016.08.26.rem1}
The inverse to $\wt{u}_{1,\Gamma}$ can be expressed by the formula 
$$\wt{u}_{1,\Gamma}^{-1}(H)=int_{\Gamma,u_{1,\Gamma}^{-1}(H\circ p)}^{-1}(Id_{int(\Gamma)}*_{H\circ p}H)$$
Note that while we can omit explicitly mentioning $dom(f)$ and $codom(f)$ when we write $int(f)$ we must specify them when we write $int^{-1}(f)$ because $int$ is bijective only on the subsets of morphisms with fixed domain and codomain. This makes the expression for $\wt{u}_{1,\Gamma}^{-1}$ longer than one would prefer.
\end{remark}
The family of functions $\partial_{\Gamma}$ forms a morphism of presheaves $\wt{Ob}_n\sr \Ob_n$ that we usually denote simply by $\partial$.  
\begin{lemma}
\llabel{2016.12.02.l4}
The square of  morphisms of presheaves
\begin{eq}\llabel{2016.08.20.eq1}
\begin{CD}
\wt{\Ob}_1 @>\wt{u}_1>> int^{\circ}(Yo(\wt{U}))\\
@V\partial VV @VV int^{\circ}(Yo(p))V\\
\Ob_1 @>u_1>> int^{\circ}(Yo(U))
\end{CD}
\end{eq}
commutes. 
\end{lemma}
\begin{proof}
For $\Gamma$ and $o\in \wt{\Ob}_1(\Gamma)$ we have 
$$int^{\circ}(Yo(p))_{\Gamma}(\wt{u}_{1,\Gamma}(o))=
(\wt{u}_{1,\Gamma}(o))\circ p=
int(o)\circ Q(u_{1,\Gamma}(\partial(o)))\circ p=
$$$$int(o)\circ p_{u_{1,\Gamma}(\partial(o))}\circ u_{1,\Gamma}(\partial(o))=
int(o\circ p_{\partial(o)})\circ u_{1,\Gamma}(\partial(o))=
u_{1,\Gamma}(\partial(o))$$
where the first equality is by definition of $int^{\circ}$ and $Yo$, the second  by (\ref{2015.04.30.eq4a}), the third by commutativity of (\ref{2016.12.02.eq8}), the fourth by (\ref{2016.08.24.eq3}) and the fifth by the definition $\wt{\Ob}_1(\Gamma)$ in (\ref{2016.11.15.eq6}) and the fact that $\partial(o)=codom(o)$. The lemma is proved. 
\end{proof}

\subsection{Functor isomorphisms $SD_p$}
\llabel{Sec.4}

In this section we continue to consider a universe category $({\cal C},p)$. For any $({\cal C},p)$ we will relate the functor $D_p$ on $PreShv({\cal C})$ and the functor $Sig$ on $PreShv(CC({\cal C},p))$.
\begin{problem}\llabel{2016.08.28.prob1}
For a universe category $({\cal C},p)$ to construct an isomorphism of functors from $PreShv({\cal C})$ to $PreShv(CC)$ of the form
$$SD_p:int^{\circ}\circ Sig \sr D_p\circ int^{\circ}$$
\end{problem}
\begin{construction}\rm\llabel{2016.08.28.constr1}
In view of Lemma \ref{2016.11.14.l1}, we have to construct, for any ${\cal G}\in PreShv({\cal C})$, an isomorphism of presheaves on $CC({\cal C},p)$ of the form
$$SD_{p,{\cal G}}:Sig(int^{op}\circ {\cal G})\sr int^{op}\circ D_p({\cal G})$$
and to show that these isomorphisms are natural in ${\cal G}$, that is, that for a morphism of presheaves $r:{\cal G}\sr {\cal G}'$ one has 
$$SD_{p,{\cal G}}\circ int^{\circ}(D_p(r))=Sig(int^{\circ}(r))\circ SD_{p,{\cal G}'}$$
Applying Lemma \ref{2016.11.14.l1} again, we see that we need to construct, for each $\cal G$ and $\Gamma\in CC({\cal C},p)$, a bijection $SD_{p,{\cal G},\Gamma}$, which we will denote $\phi_{{\cal G},\Gamma}$ for the duration of the proof,  of the form  
$$\phi_{{\cal G},\Gamma}:Sig(int^{op}\circ {\cal G})(\Gamma)\sr (int^{op}\circ D_p({\cal G}))(\Gamma)=D_p({\cal G})(int(\Gamma))$$
and to show that two conditions hold:
\begin{enumerate}
\item for any $f:\Gamma'\sr \Gamma$ we have
\begin{eq}
\llabel{2016.08.30.eq1}
\phi_{{\cal G},\Gamma}\circ D_p({\cal G})(int(f))=Sig(int^{op}\circ {\cal G})(f)\circ \phi_{{\cal G},\Gamma'}
\end{eq}
that is, the square
\begin{eq}
\llabel{2016.11.19.eq1}
\begin{CD}
Sig(int^{op}\circ {\cal G})(\Gamma) @>\phi_{{\cal G},\Gamma}>> D_p({\cal G})(int(\Gamma))\\
@VSig(int^{op}\circ {\cal G})(f)VV @VVD_p({\cal G})(int(f))V\\
Sig(int^{op}\circ {\cal G})(\Gamma') @>\phi_{{\cal G},\Gamma'}>> D_p({\cal G})(int(\Gamma'))
\end{CD}
\end{eq}
commutes.
\item for any $r:{\cal G}\sr {\cal G}'$ and $\Gamma\in CC({\cal C},p)$ we have
\begin{eq}
\llabel{2016.08.30.eq2}
\phi_{{\cal G},\Gamma}\circ D_p(r)_{int(\Gamma)}=Sig(int^{\circ}(r))_{\Gamma}\circ \phi_{{\cal G}',\Gamma}
\end{eq}
that is, the square
\begin{eq}
\llabel{2016.11.19.eq2}
\begin{CD}
Sig(int^{op}\circ {\cal G})(\Gamma) @>\phi_{{\cal G},\Gamma}>> D_p({\cal G})(int(\Gamma))\\
@VSig(int^{\circ}(r))_{\Gamma}VV @VVD_p(r)_{int(\Gamma)}V\\
Sig(int^{op}\circ {\cal G}')(\Gamma) @>\phi_{{\cal G},\Gamma}>> D_p({\cal G}')(int(\Gamma))
\end{CD}
\end{eq}
commutes.
\end{enumerate}
To construct $\phi_{{\cal G},\Gamma}$ we first compute using (\ref{2016.08.30.eq7})
$$Sig(int^{op}\circ {\cal G})(\Gamma)=\amalg_{T\in\Ob_1(\Gamma)}{\cal G}(int(T))$$
and using (\ref{2016.08.30.eq4})
$$D_p({\cal G})(int(\Gamma))=\amalg_{F:int(\Gamma)\sr U}{\cal G}((int(\Gamma);F))$$
and define the function $\phi_{{\cal G},\Gamma}$ by the formula
\begin{eq}
\llabel{2016.09.01.eq3}
\phi_{{\cal G},\Gamma}((T,g))=(u_{1,\Gamma}(T),g)
\end{eq}
where the right hand side is defined because of (\ref{2015.05.02.eq1a}).  The function $\phi_{{\cal G},\Gamma}$ is a bijection as the total function of the bijection $u_{1,\Gamma}$ and the family of bijections, namely the identity functions.

To prove equality (\ref{2016.08.30.eq1}) we compute using (\ref{2016.08.30.eq8})
$$Sig(int\circ {\cal G})(f)(T,g)=(f^*(T),{\cal G}(int(q(f,T)))(int(T)))$$
and using (\ref{2016.08.30.eq5})
$$D_p({\cal G})(int(f))(F,g)=(int(f)\circ F,{\cal G}(Q(int(f),F))(g))$$
Equality (\ref{2016.08.30.eq1}) follows now from (\ref{2015.04.30.eq1a}) and (\ref{2016.08.30.eq3}). 

To prove equality (\ref{2016.08.30.eq2}) we compute using (\ref{2016.08.30.eq9})
$$Sig(int^{\circ}(r))_{\Gamma}(T,g)=(T,r_{int(T)}(g))$$
and using (\ref{2016.08.30.eq6})
$$D_p(r)_{int(\Gamma)}(F,g)=(F,r_{(int(\Gamma);F)}(g))$$
and (\ref{2016.08.30.eq2}) follows from (\ref{2015.05.02.eq1a}).

This completes Construction \ref{2016.08.28.constr1}.
\end{construction}

\subsection{Isomorphisms of presheaves $u_n$ and $\wt{u}_n$ for $n\ge 2$}
\llabel{Sec.5}

In this section we continue to consider a universe category $({\cal C},p)$. For any such $({\cal C},p)$ and any $n\ge 1$, we construct isomorphisms of presheaves on $CC({\cal C},p)$ of the form
\begin{eq}
\llabel{2016.11.22.eq1}
u_n:\Ob_n\sr int^{\circ}(D_p^{n-1}(Yo(U)))
\end{eq}
and
\begin{eq}
\llabel{2016.11.22.eq2}
\wt{u}_n:\wt{\Ob}_n\sr int^{\circ}(D_p^{n-1}(Yo(\wt{U})))
\end{eq}
where $D_p^0=Id_{PreShv({\cal C})}$, and $u_1$ and $\wt{u}_1$ are the isomorphisms constructed in Section \ref{Sec.3}. We show that 
\begin{eq}
\llabel{2016.12.02.eq7}
\wt{u}_n\circ int^{\circ}(D_p^{n-1}(Yo(p)))=\partial\circ \wt{u}_n
\end{eq}

Let us fix a universe category $({\cal C},p)$.
\begin{problem}
\llabel{2016.11.22.prob1}
Let $n\ge 2$. To construct an isomorphism of presheaves on $CC({\cal C},p)$ of the form (\ref{2016.11.22.eq1}). 
\end{problem}
\begin{construction}\rm
\llabel{2016.11.22.constr1}
We proceed by induction on $n$ starting with $n=1$. Observe that $SD_{p,{\cal G}}$ is an isomorphism of the form
\begin{eq}
\llabel{2016.11.22.eq3}
Sig(int^{\circ}({\cal G}))\sr int^{\circ}(D_p({\cal G}))
\end{eq}
The isomorphism $u_1$ was constructed in Section \ref{Sec.3}. For the successor, define $u_{n+1}$ as the following composition
$$
\begin{CD}
\Ob_{n+1} @>S\Ob_{n}^{-1}>> Sig(\Ob_n) @>Sig(u_n)>> Sig(int^{\circ}(D_p^{n-1}(Yo(U)))) @>SD_{p,D_p^{n-1}(Yo(U))}>> \\
@. int^{\circ}(D_p(D_p^{n-1}(Yo(U)))) @= int^{\circ}(D_p^n(Yo(U)))
\end{CD}
$$
\end{construction}
The isomorphism $u_{n+1,\Gamma}$ is of the form
\begin{eq}
\llabel{2016.12.22.eq1}
T\mapsto (ft^n(T),T)\mapsto (ft^n(T),u_{n,ft^n(T)}(T))\mapsto (u_{1,\Gamma}(ft^n(T)),u_{n,ft^n(T)}(T))
\end{eq}
where the form of the first map is by (\ref{2016.09.01.eq5}), the second by (\ref{2016.08.30.eq9}) and the third by (\ref{2016.09.01.eq3}). In particular,
for $n=1$ we get
$$u_{2,\Gamma}(T)=(u_{1,\Gamma}(ft(T)),u_{1,ft(T)}(T))$$
\begin{problem}
\llabel{2016.11.22.prob2}
Let $n\ge 2$. To construct an isomorphism of presheaves on $CC({\cal C},p)$ of the form (\ref{2016.11.22.eq2}).
\end{problem}
\begin{construction}\rm
\llabel{2016.11.22.constr2}
We proceed by induction on $n$ starting with $n=1$. The isomorphism $\wt{u}_1$ was constructed in Section \ref{Sec.3}. For the successor, define $\wt{u}_{n+1}$ as the following composition, where we use that $SD_{p,{\cal G}}$ is of the form (\ref{2016.11.22.eq3}),
$$
\begin{CD}
\wt{\Ob}_{n+1} @>S\wt{\Ob}_{n}^{-1}>> Sig(\wt{\Ob}_n) @>Sig(\wt{u}_n)>> Sig(int^{\circ}(D_p^{n-1}(Yo(\wt{U})))) @>SD_{p,D_p^{n-1}(Yo(\wt{U}))}>> \\
@. int^{\circ}(D_p(D_p^{n-1}(Yo(\wt{U})))) @= int^{\circ}(D_p^n(Yo(\wt{U})))
\end{CD}
$$
\end{construction}
The isomorphism $\wt{u}_{n+1,\Gamma}$ is of the form
\begin{eq}
\llabel{2016.12.22.eq2}
o\mapsto (ft^n(\partial(o)),o)\mapsto (ft^n(\partial(o)),\wt{u}_{n,ft^n(\partial(o))}(o))\mapsto (u_{1,\Gamma}(ft^n(\partial(o))),\wt{u}_{n,ft^n(\partial(o))}(o))
\end{eq}
where the form of the first map is by (\ref{2016.09.01.eq7}), the second by (\ref{2016.08.30.eq9}) and the third by (\ref{2016.09.01.eq3}). In particular,
for $n=1$ we get
$$\wt{u}_{2,\Gamma}(o)=(u_{1,\Gamma}(ft(\partial(o))),\wt{u}_{1,ft(\partial(o))}(o))$$
\begin{lemma}
\llabel{2016.12.02.l3}
For any $n\ge 1$, (\ref{2016.12.02.eq7}) holds, that is, the square
$$
\begin{CD}
\wt{\Ob}_n @>\wt{u}_n>> int^{\circ}(D_p^{n-1}(Yo(\wt{U})))\\
@V\partial VV @VVint^{\circ}(D_p^{n-1}(Yo(p)))V\\
\Ob_n@>u_n>> int^{\circ}(D_p^{n-1}(Yo(U)))
\end{CD}
$$
commutes.
\end{lemma}
\begin{proof}
We proceed by induction on $n$ starting at $n=1$. For $n=1$ it is shown in Lemma \ref{2016.12.02.l4}.

For the successor of $n$ we have the diagram
$$
\begin{CD}
\wOb_{n+1} @>S\wOb_n^{-1}>> Sig(\wOb_n) @>Sig(\wt{u}_n)>> Sig(int^{\circ}(D_p^{n}(Yo(\wt{U})))) @>SD_p>> int^{\circ}(D_p(D_p^{n}(Yo(\wt{U}))))\\
@V\partial VV @VSig(\partial) VV @VSig(int^{\circ}(D_p^n(Yo(p)))) VV @Vint^{\circ}(D_p(D_{p}^n(Yo(p)))) VV\\
\Ob_{n+1} @>S\Ob_n^{-1}>> Sig(\Ob_n) @>Sig(u_n)>> Sig(int^{\circ}(D_p^{n}(Yo(U)))) @>SD_p>> int^{\circ}(D_p(D_p^{n}(Yo(U))))
\end{CD}
$$
where the composition of the upper horizontal arrows is $\wt{u}_n$ and the composition of the lower horizontal ones is $u_n$. To prove the lemma it is sufficient to show that the three squares of the diagram commute. 

The commutativity of the left square follows easily from  Lemma \ref{2016.12.04.l1}. The middle square commutes by the inductive assumption using the fact that $Sig$ is a functor. The right square commutes because $SD_p$ is an isomorphism of functors, that is, it is natural in morphisms of presheaves. 
\end{proof}

\subsection{The case of a locally cartesian closed $\cal C$ - isomorphisms $\eta_n$ and $\mu_n$}
\llabel{Sec.6}
In this section $\cal C$ is a locally cartesian closed category (see Appendix \ref{App.2}) with a binary product structure (see Appendix \ref{App.1}). 

The main construction of this section is Construction \ref{2015.03.29.constr1}  for Problem \ref{2015.03.29.prob1} that provides, for a universe $p$ in a category $\cal C$ as above, representations for the presheaves $D_p(Yo(Y))$.  As a corollary we provide constructions for Problems \ref{2016.12.02.prob1} and \ref{2015.03.17.prob3}.

For a morphism $p:\wt{U}\sr U$ in $\cal C$ and an object $Y$ of $\cal C$  let 
$$I_p(Y):=\uu{Hom}_U((\wt{U},p),(U\times Y,pr_1))$$
and let 
$$prI_p(Y)=p\triangle pr_1:I_p(Y)\sr U$$
be the canonical morphism.

For a morphism $f:Y\sr Y'$ let
$$I_p(f)=\uu{Hom}_U((\wt{U},p),U\times f)$$
By (\ref{2016.11.30.eq1}),(\ref{2016.11.30.eq2}) and Definition \ref{2016.11.28.def1}(3) we have
$$I_p(Id_Y)=Id_{I_p(Y)}$$
and for $f':Y'\sr Y''$ we have
$$I_p(f\circ f')=I_p(f)\circ I_p(f')$$
which shows that the mappings $Y\mapsto I_p(Y)$ and $f\mapsto I_p(f)$ define a functor from $\cal C$ to itself.

The main goal of this section is to construct an isomorphism $\eta$ between functors from ${\cal C}$ to $PreShv({\cal C})$ of the form:
$$\eta:Yo\circ D_p\sr I_p\circ Yo$$

This isomorphism provides, in particular, a family, parametrized by $Y\in {\cal C}$,  of representations for the functors $D_p(Yo(Y))$.

Note that $I_p$ depends on the choice of both the locally cartesian closed and the binary product structures on $\cal C$, but does not depend on the universe structure. On the other hand, the construction of the functors $D_p(F)$ requires a universe structure on $p$ but does not require either the locally cartesian closed or the binary product structure on $\cal C$. 

The computations below are required because we have to deal with this fact. In particular, we have to take into the account that for $F:X\sr U$ the fiber product $(X,F)\times_U(\wt{U},p)$ that we have from the structure of a category with pullbacks on $\cal C$ need not be equal to $(X;F)$ that we have from the universe structure on $p$. 

Let $p:\wt{U}\sr U$ be a universe and $Y$ an object of $\cal C$. We assume that $\cal C$ is equipped with a locally cartesian closed and a binary product structures. For $F:X\sr U$ there is a unique morphism
$$\iota_F:(X;F)\sr (X,F)\times_U(\wt{U},p)$$
such that 
\begin{eq}
\llabel{2016.12.02.eq3}
\begin{array}{c}
\iota_F\circ pr_1=p_{F}\\\\
\iota_F\circ pr_2=Q(F)
\end{array}
\end{eq}
which is a particular case of the morphisms $\iota$ of Lemma \ref{2015.04.16.l1}. 

The evaluation morphism in the case of $I_p(Y)$ is a morphism in ${\cal C}/U$ of the form 
$$evI_p: (I_p(Y),prI_p(Y))\times_U(U\times Y, pr_1)\sr (U\times Y,pr_1)$$
Define a morphism
$$st_p(Y):(I_p(Y);prI_p(Y))\sr Y$$
as the composition:
\begin{eq}
\llabel{2016.12.02.eq2}
st_p(Y):=\iota_{prI_p(Y)}\circ evI_p(Y)\circ pr_2
\end{eq}
We will need to use some properties of these morphisms.
\begin{lemma}
\llabel{2015.04.14.l2a}
Let $f:Y\sr Y'$ be a morphism, then one has
$$Q(I_p(f),prI_p(Y'))\circ st_p(Y')=st_p(Y)\circ f$$
\end{lemma}
\begin{proof}
Let $pr=prI_p(Y)$, $pr'=prI_{p}(Y')$, $\iota=\iota_{pr}$, $\iota'=\iota_{pr'}$, $ev=evI_p(Y)$ and $ev'=evI_p(Y')$. Then we have to verify that the outer square of the following diagram commutes:
$$
\begin{CD}
(I_p(Y);pr) @>\iota>> (I_p(Y),pr)\times_U(\wt{U},p) @>ev>> U\times Y @>pr_2 >> Y\\
@VQ(I_p(f),pr') VV @V I_p(f)\times Id_{\wt{U}} VV@V Id_U\times f VV @VV f V\\
(I_p(Y');pr') @>\iota'>> (I_p(Y'),pr')\times_U(\wt{U},p) @>ev'>> U\times Y' @>pr_2 >> Y'
\end{CD}
$$
The commutativity of the left square is a particular case of Lemma \ref{2015.04.16.l1}. The commutativity of the right square is an immediate corollary of the definition of $Id_U\times f$.  The commutativity of the middle square is a particular case of (\ref{2016.11.28.eq1}). 
\end{proof}
\begin{remark}\rm
\llabel{2016.04.23.rem1}
In \cite{GambinoHyland} generalized polynomial functors are defined as functors isomorphic to functors of the form $I_p$. 
\end{remark}
\begin{problem}
\llabel{2015.03.29.prob1}
Let $\cal C$ be a locally cartesian closed category with a binary product structure and $p$ a universe in $\cal C$. To construct, for all $Y\in{\cal C}$, isomorphisms of presheaves
$$\eta_Y:D_p(Yo(Y))\sr Yo(I_p(Y))$$
that are natural in $Y$, i.e., such that for all  $f:Y\sr Y'$ the square
$$
\begin{CD}
D_p(Yo(Y)) @>D_p(Yo(f))>> D_p(Yo(Y'))\\
@V\eta_{Y} VV @VV\eta_{Y'}V\\
Yo(I_p(Y)) @>Yo(I_p(f))>> Yo(I_p(Y'))
\end{CD}
$$
commutes.
\end{problem}
\begin{construction}\rm
\llabel{2015.03.29.constr1}
We will use the notation introduced before Remark \ref{2015.07.29.rem2}. We need to construct bijections 
$$\eta_{Y,X}:D_p(X,Y)\sr Mor_{\cal C}(X,I_p(Y))$$
such that for all $f:Y\sr Y'$, $X\in{\cal C}$ and $d\in D_p(X,Y)$ one has
\begin{eq}
\llabel{2016.09.11.eq1}
\eta_{Y,X}(d)\circ I_p(f)=\eta_{Y',X}(d\circ f)
\end{eq}
and for any $f:X'\sr X$ and $d\in D_p(Yo(Y))(X)$ one has
\begin{eq}
\llabel{2016.09.11.eq2}
f\circ \eta_{Y,X}(d)=\eta_{Y,X'}(f\circ d)
\end{eq}
We will construct bijections 
$$\eta^{!}_{Y,X}:Mor(X,I_p(Y))\sr D_p(X,Y)$$
such that for all $g:X\sr I_p(Y)$ one has:
\begin{enumerate}
\item for all $f:Y\sr Y'$ one has 
\begin{eq}
\llabel{2016.09.11.eq3}
\eta^!(g)\circ f=\eta^!(g\circ I_p(f))
\end{eq}
\item for all $f:X'\sr X$ one has 
\begin{eq}
\llabel{2016.09.11.eq4}
f\circ \eta^!(g)=\eta^!(f\circ g)
\end{eq}
\end{enumerate}
and then define $\eta_{Y,X}$ as the inverse to $\eta^!_{Y,X}$. One proves easily that (\ref{2016.09.11.eq1}) implies (\ref{2016.09.11.eq3}) and (\ref{2016.09.11.eq2}) implies (\ref{2016.09.11.eq4}).

By (\ref{2016.11.14.eq4}) we have
$$D_p(X,Y)=\amalg_{F:X\sr U}Mor_{\cal C}((X;F),Y)$$
For $g:X\sr I_p(Y)$ we set
\begin{eq}
\llabel{2016.12.02.eq5}
\eta^{!}_{Y,X}(g):=(g\circ prI_p(Y), Q(g,prI_p(Y))\circ st_p(Y))
\end{eq}
as can be seen on the diagram
$$
\begin{CD}
@. Y\\
@. @AAst_p(Y)A\\
(X;g\circ prI_p(Y)) @>Q(g,prI_p(Y))>> (I_p(Y);prI_p(Y)) @>Q(prI_p(Y))>> \wt{U}\\
@VVV @VVV @VVpV\\
X @>g>> I_p(Y) @>prI_p(Y)>> U
\end{CD}
$$

To see that this is a bijection observe first that it equals the composition
$$Mor(X,I_p(Y))\sr \amalg_{F:X\sr U}Mor_U((X,F),(I_p(Y),prI_p(Y)))\sr \amalg_{F:X\sr U}Mor((X;F),Y)$$
where the first function is given by the formula $g\mapsto (g\circ prI_p(Y),g)$  and the second is the sum over all $F:X\sr U$ of functions $g\mapsto Q(g,prI_p(Y))\circ st_p(Y)$. 

The first function is a function of the form $A\sr \amalg_{b\in B}h^{-1}(b)$, which is defined and is a bijection for any function of sets $h:A\sr B$. It remains to show that the second one is a bijection for every $F$.

By definition of the $\uu{Hom}_U$ structure we know that for each $F$ the function
$$adj:Mor_U((X,F),(I_{p}(Y),prI_p(Y)))\sr Mor_U((X,F)\times_U(\wt{U},p),(U\times Y,pr_1))$$
given by $g\mapsto (g\times_U Id_{(\wt{U},p)})\circ evI_p(Y)$ is a bijection. 

By definition of the binary product, the function of post-composition with $pr_2$, 
$$Mor_U((X,F)\times_U(\wt{U},p),(U\times Y,pr_1))\sr Mor((X,F)\times_U(\wt{U},p), Y)$$
is a bijection. By Lemma \ref{2016.12.02.l1}, $\iota_F$ is an isomorphism
and therefore the pre-composition with it is a bijection. Now we have two functions
$$Mor_U((X,F),(I_{p}(Y),prI_p(Y)))\sr Mor((X;F),Y)$$
given by $g\mapsto \iota_F\circ (g\times_U Id_{\wt{U}})\circ evI_p(Y)\circ pr_2$ and $g\mapsto Q(g,prI_p(Y))\circ st_p(Y)$ of which the first one is the bijection. It remains to show that these functions are equal. In view of (\ref{2016.12.02.eq2}) it is sufficient to show that 
$$\iota_F\circ(g\times_U Id_{\wt{U}})=Q(g,prI_p(Y))\circ \iota_{prI_p(Y)}$$
To do it we have to to show that the compositions of the left and right hand sides with $pr_1$ (to $I_p(Y)$) and $pr_2$ (to $\wt{U}$) are equal.

For $pr_1$ we have
$$\iota_F\circ(g\times_U Id_{\wt{U}})\circ pr_1=\iota_F\circ pr_1\circ g=p_F\circ g$$
$$Q(g,prI_p(Y))\circ \iota_{prI_p(Y)}\circ pr_1=Q(g,prI_p(Y))\circ p_{prI_p(Y)}=p_{g\circ prI_p(Y)}\circ g=p_F\circ g$$
where we used the defining equations (\ref{2016.12.02.eq3}) of $\iota$, the definition (\ref{2016.12.02.eq4}) of $Q(-,-)$ and the fact that $g$ is a morphism over $U$. 

For $pr_2$ we have
$$\iota_F\circ(g\times_U Id_{\wt{U}})\circ pr_2=\iota_F\circ pr_2\circ Id_{\wt{U}}=\iota_F\circ pr_2=Q(F)$$
$$Q(g,prI_p(Y))\circ \iota_{prI_p(Y)}\circ pr_2=Q(g,prI_p(Y))\circ Q(prI_p(Y))=Q(g\circ prI_p(Y))=Q(F)$$
where we used the defining equations (\ref{2016.12.02.eq3}) of $\iota$, (\ref{2016.08.24.eq4}) and the fact that $g$ is a morphism over $U$.

We now have to check the behavior of $\eta^!$ with respect to morphisms in $Y$ (equality (\ref{2016.09.11.eq3})) and $X$ (equality (\ref{2016.09.11.eq4}).

Let $pr=prI_p(Y)$ and $pr'=prI_p(Y')$. Let $g:X\sr I_p(Y)$ be as above. For $f:Y\sr Y'$ we have
$$\eta^!(g)\circ f=D_p(Yo(f))_X(g\circ pr, Q(g,pr)\circ st_p(Y))=(g\circ pr, Q(g,pr)\circ st_p(Y)\circ f)$$
where the first equality is by (\ref{2016.12.02.eq5}) and the second by (\ref{2016.12.02.eq6}) and
$$\eta^!(g\circ I_p(f))=(g\circ I_p(f)\circ pr', Q(g\circ I_p(f),pr')\circ st_p(Y'))$$
where the equality is by (\ref{2016.12.02.eq5}). We have $pr=I_p(f)\circ pr'$ because $I_p(f)$ is a morphism over $U$. It remains to check that 
$$Q(g,pr)\circ st_p(Y)\circ f=Q(g\circ I_p(f),pr')\circ st_p(Y')$$
By \cite[Lemma 2.5]{Cfromauniverse} we have
$$Q(g\circ I_p(f),pr')=Q(g,pr)\circ Q(I_p(f),pr')$$
and the remaining equality
$$Q(g,pr)\circ st_p(Y)\circ f=Q(g,pr)\circ Q(I_p(f),pr')\circ st_p(Y')$$
follows from Lemma \ref{2015.04.14.l2a}.

Consider now $f:X'\sr X$. Then 
$$f\circ \eta^!(g)=D_p(Yo(Y))(f)(g\circ pr, Q(g,pr)\circ st_p(Y))=$$$$(f\circ g\circ pr,  Q(f, g\circ pr)\circ Q(g,pr)\circ st_p(Y))$$
and
$$\eta^!(f\circ g)=(f\circ g\circ pr, Q(f\circ g, pr)\circ st_p(Y))$$
where we used (\ref{2016.12.02.eq5}) and (\ref{2016.11.14.eq4a}) 
and the required equality follows from \cite[Lemma 2.5]{Cfromauniverse}.  
\end{construction}

\begin{problem}
\llabel{2016.12.02.prob1}
For a locally cartesian closed category $\cal C$ with a binary product structure and a universe $p$ in $\cal C$ to construct, for all $n\ge 0$ and $Y\in {\cal C}$, isomorphisms of presheaves
$$\eta_{n,Y}:D_p^n(Yo(Y))\sr Yo(I_p^n(Y))$$
that are natural in $Y$, i.e., such that for all  $f:Y\sr Y'$ the square
\begin{eq}
\llabel{2017.01.03.eq1}
\begin{CD}
D_p^n(Yo(Y)) @>D_p^n(Yo(f))>> D_p^n(Yo(Y'))\\
@V\eta_{n,Y} VV @VV\eta_{n,Y'}V\\
Yo(I_p^n(Y)) @>Yo(I_p^n(f))>> Yo(I_p^n(Y'))
\end{CD}
\end{eq}
commutes.
\end{problem}
\begin{construction}
\llabel{2016.12.02.constr1}\rm
Proceed by induction on $n$ starting with $n=0$. By our convention, $D_p^0=Id_{PreShv({\cal C})}$ and $I_p^0=Id_{\cal C}$. We set $\eta_{0,Y}=Id_{Yo(Y)}$. For the successor we define $\eta_{n+1,Y}$ as the composition
$$
\begin{CD}
@. D_p^{n+1}(Yo(Y))=\\
D_p(D_p^n(Yo(Y)))@>D_p(\eta_{n,Y})>> D_p(Yo(I_p^n(Y))) @>\eta_{1,I_p^n(Y)}>> Yo(I_p(I_p^n(Y)))=\\
@. Yo(I_p^{n+1}(Y))
\end{CD}
$$
The naturality in $Y$ is easily proved by induction.
\end{construction}
Note that we can write $\eta_{n,Y,X}$ as a function of the form
$$D_p^n(X,Y)\sr Mor_{\cal C}(X,I_p^n(Y))$$
Let us spell out the formulas expressing the fact that $\eta_{n,Y}$ is a morphism of presheaves and the naturality of $\eta_{n,Y}$ in $Y$ in the $\circ$-notation. Let $d\in D_p^n(X,Y)$. Then for $f:X'\sr X$ one has
\begin{eq}
\llabel{2017.01.03.eq2}
\eta_n(f\circ d)=f\circ \eta_n(d)
\end{eq}
and for $g:Y\sr Y'$ one has
\begin{eq}
\llabel{2017.01.03.eq3}
\eta_n(d\circ g)=\eta_n(d)\circ I_p^n(g)
\end{eq}
Indeed, the first formula is an expression of the fact that the family of functions $\eta_{n,Y,-}$ is a morphism of presheaves and the second formula an expression of the commutativity of the square (\ref{2017.01.03.eq1}). 

We let $\eta_{n,Y}^!$ denote the isomorphism inverse to $\eta_{n,Y}$. For $m:X\sr I_p^n(Y)$ we have the following formulas that follow from (\ref{2017.01.03.eq2}) and (\ref{2017.01.03.eq3}). For $f:X\sr X'$ one has 
\begin{eq}
\llabel{2016.09.11.eq4n}
f\circ \eta_n^!(m)=\eta_n^!(f\circ m)
\end{eq}
and for $g:Y\sr Y'$ one has 
\begin{eq}
\llabel{2016.09.11.eq3n}
\eta_n^!(m)\circ g=\eta_n^!(m\circ I_p^n(g))
\end{eq}
Let us also introduce the following notation that will be useful below. For $Y\in {\cal C}$ let 
\begin{eq}
\llabel{2017.01.07.eq1}
Id^n_Y=\eta_n^!(Id_{I_p^n(Y)})\in D_p^n(I_p^n(Y),Y)
\end{eq}
We have the following formulas.
\begin{lemma}
\llabel{2017.01.07.l2}
In the notations introduced above one has:
\begin{enumerate}
\item for $m:X\sr I_p^n(Y)$ one has
\begin{eq}
\llabel{2017.01.07.eq3}
m\circ Id^n_Y=\eta_n^!(m)
\end{eq}
\item for $g:Y\sr Y'$ one has
\begin{eq}
\llabel{2017.01.07.eq4}
Id^n_Y\circ g=\eta_n^!(I_p^n(g))
\end{eq}
\end{enumerate}
\end{lemma}
\begin{proof}
For the first formula we have
$$m\circ Id^n_Y=m\circ \eta_n^!(Id_{I_p^n(Y)})=\eta_n^!(m\circ Id_{I_p^n(Y)})=\eta_n^!(m)$$
where the first equality is by the definition of $Id^n_Y$, the second by (\ref{2016.09.11.eq4n}) and the third by the identity axiom of $\cal C$.

For the second formula we have 
$$Id_n^Y\circ g=\eta_n^!(Id_{I_p^n(Y)})\circ g=\eta_n^!(Id_{I_p^n(Y)}\circ I_p^n(g))=\eta_n^!(I_p^n(g))$$
where the first equality is by the definition of $Id^n_Y$, the second by (\ref{2016.09.11.eq3n}) and the third by the identity axiom of $\cal C$. The lemma is proved.
\end{proof}
Note that (\ref{2017.01.07.eq3}) implies in particular that we have
\begin{eq}
\llabel{2017.01.07.eq5}
\eta_n(d)\circ Id^n_Y=\eta_n^!(\eta_n(d))=d
\end{eq}
\begin{problem}
\llabel{2015.03.17.prob3}
For $\cal C$ as above, a universe $p:\wt{U}\sr U$ in $\cal C$ and $n\ge 1$ to construct isomorphisms of presheaves
$$\mu_n: \Ob_n\sr int^{\circ}(Yo(I^{n-1}_p(U)))$$
and
$$\wt{\mu}_n: \wOb_n\sr int^{\circ}(Yo(I^{n-1}_p(\wt{U})))$$
such that the square
\begin{eq}
\llabel{2016.12.04.eq2}
\begin{CD}
\wOb_n@>\wt{\mu}_n>> int^{\circ}(Yo(I^{n-1}_p(\wt{U})))\\
@V\partial VV @VV int^{\circ}(Yo(I^{n-1}_p(p))) V\\
\Ob_n @>\mu_n>> int^{\circ}(Yo(I^{n-1}_p(U)))
\end{CD}
\end{eq}
commutes.
\end{problem}
\begin{construction}
\llabel{2015.03.17.constr2}\rm
Compose isomorphism $u_n$ of Construction \ref{2016.11.22.constr1} (resp. isomorphism $\wt{u}_n$ of Construction \ref{2016.11.22.constr2}) with the isomorphism $int^{\circ}(\eta_{n-1,U})$ (resp. $int^{\circ}(\eta_{n-1,\wt{U}})$) where $\eta_{n-1,U}$ (resp. $\eta_{n-1,\wt{U}}$) is the isomorphism of Construction \ref{2016.12.02.constr1}. 

To prove the commutativity of (\ref{2016.12.04.eq2}) consider the diagram
$$
\begin{CD}
\wOb_n @>\wt{u}_n>> int^{\circ}(D_p^{n-1}(Yo(\wt{U}))) @>int^{\circ}(\eta_{n-1,\wt{U}})>> int^{\circ}(Yo(I^{n-1}_p(\wt{U})))\\
@V\partial VV @Vint^{\circ}(D_p^{n-1}(Yo(p))) VV @VV int^{\circ}(Yo(I^{n-1}_p(p))) V\\
\Ob_n @>u_n>> int^{\circ}(D_p^{n-1}(Yo(U))) @>int^{\circ}(\eta_{n-1,U})>>  int^{\circ}(Yo(I^{n-1}_p(U)))
\end{CD}
$$
The composition of the upper arrows is $\wt{\mu}_n$ and the composition of the lower ones is $\mu_n$. It remains to show that the two squares commute. The left square commutes by Lemma \ref{2016.12.02.l3}. The right square commutes because $int^{\circ}$ is a functor and $\eta_{n-1,Y}$ is natural in $Y$.
\end{construction}
Observe that for $\Gamma\in CC({\cal C},p)$, $T\in \Ob_n(\Gamma)$ and $o\in \wOb_n(\Gamma)$ one has:
\begin{eq}
\llabel{2017.01.03.eq4}
\mu_{n,\Gamma}(T)=\eta_{n-1,U,int(\Gamma)}(u_{n,\Gamma}(T))\in int^{\circ}(Yo(I_p^{n-1}(U)))(\Gamma)=Mor_{\cal C}(int(\Gamma),I_p^{n-1}(U))
\end{eq}
and
\begin{eq}
\llabel{2017.01.03.eq5}
\wt{\mu}_{n,\Gamma}(o)=\eta_{n-1,\wt{U},int(\Gamma)}(\wt{u}_{n,\Gamma}(o))\in int^{\circ}(Yo(I_p^{n-1}(\wt{U})))(\Gamma)=Mor_{\cal C}(int(\Gamma),I_p^{n-1}(\wt{U}))
\end{eq}
and the commutativity of (\ref{2016.12.04.eq2}) is equivalent to the assertion that for all $\Gamma$ and $o$ as above one has
\begin{eq}
\llabel{2017.01.03.eq6}
\mu_{n,\Gamma}(\partial(o))=\wt{\mu}_{n,\Gamma}(o)\circ I_p^n(p)
\end{eq}

\section{Functoriality}
\subsection{Universe category functors and the $D_p$ construction}
\llabel{Sec.7}

Let $({\cal C},p,pt)$ and $({\cal C}',p',pt')$ be two universe categories. Recall from \cite{Cfromauniverse} the following definition.
\begin{definition}
\llabel{2016.12.09.def1}
A universe category functor from $({\cal C},p,pt)$ to  $({\cal C}',p',pt')$  is a triple ${\bf\Phi}=(\Phi,\phi,\wt{\phi})$ where $\Phi$ is a functor ${\cal C}\sr {\cal C}'$ and $\phi:\Phi(U)\sr U'$, $\wt{\phi}:\Phi(\wt{U})\sr \wt{U}'$ are two morphisms such that one has:
\begin{enumerate}
\item $\Phi$ takes $pt$ to a final object,
\item $\Phi$ takes the canonical pullbacks based on $p$ to pullbacks,
\item the square
\begin{eq}
\llabel{2015.03.21.sq1}
\begin{CD}
\Phi(\wt{U}) @>\wt{\phi}>> \wt{U}'\\
@V\Phi(p)VV @VVp' V\\
\Phi(U) @>\phi>> U'
\end{CD}
\end{eq}
is a pullback.
\end{enumerate}
\end{definition}

\begin{problem}
\llabel{2016.12.14.prob1}
Let ${\bf\Phi}=(\Phi,\phi,\wt{\phi})$ be a universe category functor $({\cal C},p)\sr ({\cal C}',p')$. To construct a functor morphism
\begin{eq}
\llabel{2016.12.14.eq3}
\Phi D:\Phi^{\circ}\circ D_p\sr D_{p'}\circ \Phi^{\circ}
\end{eq}
\end{problem}
\begin{construction}\rm
\llabel{2016.12.14.constr1}
Both the left and the right hand side of (\ref{2016.12.14.eq3}) are functors of the form
$$PreShv({\cal C}')\sr PreShv({\cal C})$$
Therefore, we need, for any presheaf ${\cal G}'$ on ${\cal C}'$ and any $X\in {\cal C}$, to construct a function
\begin{eq}
\llabel{2016.12.14.eq4}
\Phi D_{{\cal G}',X}:D_p(\Phi^{\circ}({\cal G}'))(X)\sr \Phi^{\circ}(D_{p'}({\cal G}'))(X)
\end{eq}
and to prove that
\begin{enumerate}
\item the family $\Phi D_{{\cal G}',-}$ is a morphism of presheaves, that is, for any $a:X\sr Y$ in $\cal C$, the square
\begin{eq}
\llabel{2016.12.16.eq2}
\begin{CD}
D_p(\Phi^{\circ}({\cal G}'))(Y) @>\Phi D_{{\cal G}',Y}>> \Phi^{\circ}(D_{p'}({\cal G}'))(Y)\\
@VD_p(\Phi^{\circ}({\cal G}'))(a)VV @VV\Phi^{\circ}(D_{p'}({\cal G}'))(a)V\\
D_p(\Phi^{\circ}({\cal G}'))(X) @>\Phi D_{{\cal G}',X}>> \Phi^{\circ}(D_{p'}({\cal G}'))(X)
\end{CD}
\end{eq}
commutes,
\item $\Phi D$ is a natural transformation of functors to presheaves, that is, for any $f':{\cal F}'\sr {\cal G}'$ and any $X\in {\cal C}$ the square
\begin{eq}
\llabel{2016.12.16.eq3}
\begin{CD}
D_p(\Phi^{\circ}({\cal F}'))(X) @>\Phi D_{{\cal F}',X}>> \Phi^{\circ}(D_{p'}({\cal F}'))(X)\\
@VD_p(\Phi^{\circ}(f'))_X VV @VV\Phi^{\circ}(D_{p'}(f'))_XV\\
D_p(\Phi^{\circ}({\cal G}'))(X) @>\Phi D_{{\cal G}',X}>> \Phi^{\circ}(D_{p'}({\cal G}'))(X)
\end{CD}
\end{eq}
commutes.
\end{enumerate}
Computing the left and right hand side of (\ref{2016.12.14.eq4}) we see that $\Phi D_{{\cal G}',X}$ should be a function of the form
$$\coprod_{F:X\sr U}{\cal G}'(\Phi((X;F)))\sr \coprod_{F':\Phi(X)\sr U'}{\cal G}'((\Phi(X);F'))$$
Let $F:X\sr U$. Consider $(\Phi(X);\Phi(F)\circ \phi)$. Since  (\ref{2015.03.21.sq1}) is a pullback there is a unique morphism $q$ such that $q\circ \wt{\phi}=Q(\Phi(F)\circ \phi)$ and $q\circ \Phi(p)=p_{\Phi(X),\Phi(F)\circ \phi}\circ \Phi(F)$.  Then the external square in the diagram
$$
\begin{CD}
(\Phi(X);\Phi(F)\circ \phi) @>q>> \Phi(\wt{U}) @>\wt{\phi}>> \wt{U}'\\
@VV p_{\Phi(X),\Phi(F)\circ \phi} V @V\Phi(p) VV @VV p' V\\
\Phi(X) @>\Phi(F)>> \Phi(U) @>\phi>> U'
\end{CD}
$$
is a pullback and since the right hand side square is a pullback, the left hand side square is a pullback as well. Together with the fact that $\Phi$ takes pullback squares based on $p$ to pullback squares this implies that we obtain two pullbacks based on $\Phi(F)$ ad $\Phi(p)$. 

By Lemma \ref{2016.12.16.l1} and Lemma \ref{2016.12.02.l1} we have a unique morphism, which is an isomorphism, 
$$\iota_{\bf\Phi}^{X,F}:(\Phi(X);\Phi(F)\circ \phi)\sr \Phi((X;F))$$
such that 
\begin{eq}
\llabel{2015.04.08.eq1}
\iota_{\bf\Phi}^{X,F}\circ \Phi(p_{X,F})=p_{\Phi(X),\Phi(F)\circ \phi}
\end{eq}
\begin{eq}
\llabel{2015.04.08.eq2}
\iota_{\bf\Phi}^{X,F}\circ \Phi(Q(F))\circ\wt{\phi}=Q(\Phi(F)\circ\phi)
\end{eq}
and we define:
\begin{eq}
\llabel{2016.12.16.eq4}
\Phi D_{{\cal G}',X}(F,\gamma')=(\Phi(F)\circ \phi, {\cal G}'(\iota_{\bf\Phi}^{X,F})(\gamma'))
\end{eq}
When no confusion is likely, we will omit the indexes at $\iota$.  

To prove that (\ref{2016.12.16.eq2}) commutes let 
$$(F:Y\sr U, \gamma'\in {\cal G}'(\Phi((Y;F))))\in D_p(\Phi^{\circ}({\cal G}'))(Y)$$
Then one path in the square gives us
$$(\Phi^{\circ}(D_{p'}({\cal G}'))(a))(\Phi D_{{\cal F}',X}((F,\gamma')))=
$$$$(\Phi^{\circ}(D_{p'}({\cal G}'))(a))((\Phi(F)\circ \phi, {\cal G}'(\iota)(\gamma'))=
D_{p'}({\cal G}')(\Phi(a))((\Phi(F)\circ \phi, {\cal G}'(\iota)(\gamma')))=
$$$$(\Phi(a)\circ \Phi(F)\circ \phi, {\cal G}'(Q(\Phi(a),\Phi(F)\circ \phi))({\cal G}'(\iota)(\gamma')))=
$$$$(\Phi(a\circ F)\circ \phi,{\cal G}'(Q(\Phi(a),\Phi(F)\circ \phi)\circ \iota)(\gamma'))$$
where the first equality is by (\ref{2016.12.16.eq4}), the second by the definition of $\Phi^{\circ}$, the third by (\ref{2016.08.30.eq5}) and the fourth by the composition axiom of $\Phi$ and ${\cal G}'$.

The other path gives us
$$\Phi D_{{\cal G}',X}(D_p(\Phi^{\circ}({\cal G}'))(a)((F,\gamma'))=
$$$$\Phi D_{{\cal G}',X}((a\circ F, \Phi^{\circ}({\cal G}')(Q(a,F))(\gamma')))=
\Phi D_{{\cal G}',X}((a\circ F, {\cal G}'(\Phi(Q(a,F)))(\gamma')))=
$$$$(\Phi(a\circ F)\circ \phi, {\cal G}'(\iota)({\cal G}'(\Phi(Q(a,F)))(\gamma')))=
$$$$(\Phi(a\circ F)\circ \phi, {\cal G}'(\iota\circ \Phi(Q(a,F)))(\gamma'))$$
where the first equality is by (\ref{2016.08.30.eq5}), the second by the definition of $\Phi^{\circ}$, the third by (\ref{2016.12.16.eq4}) and the fourth by the composition axiom of ${\cal G}'$. 

It remains to show that
\begin{eq}
\llabel{2016.12.16.eq7}
Q(\Phi(a),\Phi(F)\circ \phi)\circ \iota=\iota\circ \Phi(Q(a,F))
\end{eq}
We have four pullbacks 
$$
\begin{CD}
(\Phi(X);\Phi(a\circ F)\circ \phi) @>Q(\Phi(a\circ F)\circ \phi)>> \wt{U}'\\
@Vp_{\Phi(X),\Phi(a\circ F)\circ \phi}VV @VVp'V\\
\Phi(X) @>\Phi(a\circ F)\circ \phi>> U'
\end{CD}
\spc\spc
\begin{CD}
(\Phi(Y);\Phi(F)\circ \phi) @>Q(\Phi(F)\circ \phi)>> \wt{U}'\\
@Vp_{\Phi(Y),\Phi(F)\circ \phi}VV @VVp'V\\
\Phi(Y) @>\Phi(F)\circ \phi>> U'
\end{CD}
$$
and
$$ 
\begin{CD}
\Phi((X;a\circ F)) @>\Phi(Q(a\circ F))\circ \wt{\phi}>> \wt{U}'\\
@V\Phi(p_{X,a\circ F})VV @VVp'V\\
\Phi(X) @>\Phi(a\circ F)\circ \phi>> U'
\end{CD}
\spc\spc
\begin{CD}
\Phi((Y;F)) @>\Phi(Q(F))\circ \wt{\phi}>> \wt{U}'\\
@V\Phi(p_{Y,F})VV @VVp'V\\
\Phi(Y) @>\Phi(F)\circ \phi>> U'
\end{CD}
$$
and a morphism $\Phi(a):\Phi(X)\sr \Phi(Y)$ such that $\Phi(a\circ F)\circ \phi=\Phi(a)\circ \Phi(F)\circ \phi$. Applying to these pullbacks Lemma \ref{2016.12.16.l1} and then applying Lemma \ref{2015.04.16.l1} we  obtain a commutative square
$$
\begin{CD}
(\Phi(X);\Phi(a\circ F)\circ \phi) @>c_1(\Phi(a),Id_{\wt{U}'})>> (\Phi(Y);\Phi(F)\circ \phi)\\
@V\iota VV @VV\iota V\\
\Phi((X;a\circ F)) @>c_2(\Phi(a),Id_{\wt{U}'})>> \Phi((Y;F))
\end{CD}
$$
To prove (\ref{2016.12.16.eq7}) it remains to show that 
\begin{eq}
\llabel{2016.12.16.eq5}
c_1(\Phi(a),Id_{\wt{U}}')=Q(\Phi(a),\Phi(F)\circ \phi)
\end{eq}
and
\begin{eq}
\llabel{2016.12.16.eq6}
c_2(\Phi(a),Id_{\wt{U}}')=\Phi(Q(a,F))
\end{eq}
In view of the definition of the morphisms $c_1,c_2$ given in Lemma \ref{2015.04.16.l1} to prove (\ref{2016.12.16.eq5}) we need to show that
$$Q(\Phi(a),\Phi(F)\circ \phi)\circ p_{\Phi(Y),\Phi(F)\circ \phi}=
p_{\Phi(X),\Phi(a\circ F)\circ \phi}\circ \Phi(a)$$
$$Q(\Phi(a),\Phi(F)\circ \phi)\circ Q(\Phi(F)\circ \phi)=Q(\Phi(a\circ F)\circ \phi)$$
The first equality follows from (\ref{2016.12.02.eq4}). The second equality follows from (\ref{2016.08.24.eq4}). In both cases we need also to use that 
$\Phi(a\circ F)=\Phi(a)\circ \Phi(F)$.

To prove (\ref{2016.12.16.eq6}) we need to show that
$$\Phi(Q(a,F))\circ \Phi(p_{Y,F})=\Phi(p_{X,a\circ F})\circ \Phi(a)$$
$$\Phi(Q(a,F))\circ \Phi(Q(F))\circ \wt{\phi}= \Phi(Q(a\circ F))\circ \wt{\phi}$$
The first equality again follows from (\ref{2016.12.02.eq4}) and the composition axiom for $\Phi$ and the second equality follows from (\ref{2016.08.24.eq4}) and the composition axiom for $\Phi$. This completes the proof of commutativity of (\ref{2016.12.16.eq2}).

To prove that (\ref{2016.12.16.eq3}) commutes let 
$$(F:X\sr U,\beta'\in {\cal F}'(\Phi((X;F))))\in D_p(\Phi^{\circ}({\cal F}'))(X)$$
Then one path in the square gives us
$$\Phi^{\circ}(D_{p'}(f'))_X(\Phi D_{{\cal F}',X}((F,\beta'))=
$$$$\Phi^{\circ}(D_{p'}(f'))_X((\Phi(F)\circ \phi,{\cal F}'(\iota)(\beta')))=
D_{p'}(f')_{\Phi(X)}((\Phi(F)\circ \phi,{\cal F}'(\iota)(\beta')))=
$$$$(\Phi(F)\circ \phi,f'_{(\Phi(X);\Phi(F)\circ f)}({\cal F}'(\iota)(\beta')))
$$
where the first equality is by (\ref{2016.12.16.eq4}), the second by the definition of $\Phi^{\circ}$ and the third by (\ref{2016.08.30.eq6}). 

The other path gives us
$$\Phi D_{{\cal G}',X}(D_p(\Phi^{\circ}(f'))_X((F,\beta')))=
$$$$\Phi D_{{\cal G}',X}((F,(\Phi^{\circ}(f'))_{(X;F)}(\beta')))=\Phi D_{{\cal G}',X}((F,f'_{\Phi((X;F))}(\beta')))=
$$$$(\Phi(F)\circ \phi,{\cal G}'(\iota)(f'_{\Phi((X;F))}(\beta')))$$
where the first equality is by (\ref{2016.08.30.eq6}), the second by the definition of $\Phi^{\circ}$ and the third by (\ref{2016.12.16.eq4}). 

It remains to show that
$$f'_{(\Phi(X);\Phi(F)\circ f)}({\cal F}'(\iota)(\beta'))={\cal G}'(\iota)(f'_{\Phi((X;F))}(\beta'))$$
which follows from the axiom of compatibility with morphisms of the natural transformation $f':{\cal F}'\sr {\cal G}'$. This completes the proof of commutativity of (\ref{2016.12.16.eq3}) and with it Construction \ref{2016.12.14.constr1}.
\end{construction}
\begin{problem}
\llabel{2016.12.18.prob1}
Let $\bf\Phi:({\cal C},p)\sr ({\cal C}',p')$ be a universe category functor. Let ${\cal F}\in PreShv({\cal C})$, ${\cal F}'\in PreShv({\cal C}')$ and let
$$m:{\cal F}\sr \Phi^{\circ}({\cal F}')$$
be a morphism of presheaves. Let $n\in\nat$. To construct a morphism of presheaves
$$D_{\bf\Phi}^n(m):D_p^n({\cal F})\sr \Phi^{\circ}(D_{p'}^n({\cal F}'))$$
\end{problem}
\begin{construction}\rm
\llabel{2016.12.18.constr1}
We proceed by induction on $n$. 

For $n=0$ we set $D_{\bf\Phi}^0(m)=m$. 

For the successor of $n$ we need to construct a morphism
$$D_{\bf\Phi}^{n+1}(m):D_p(D_p^n({\cal F}))\sr \Phi^{\circ}(D_{p'}(D_{p'}^n({\cal F}')))$$
We define it as the composition
$$
\begin{CD}
D_p(D_p^n({\cal F})) @>D_p(D_{\bf\Phi}^n(m))>> D_p(\Phi^{\circ}(D_{p'}^n({\cal F}'))) @>\Phi D_{D_{p'}^n({\cal F}')}>> \Phi^{\circ}(D_{p'}(D_{p'}^n({\cal F}')))
\end{CD}
$$
\end{construction}
The explicit form of the morphism $D_p^n(m)$ when $n\ge 1$ is given by the following lemma.
\begin{lemma}
\llabel{2016.12.22.l1}
In the context of Problem \ref{2016.12.18.prob1}, let $n\ge 1$, $X\in {\cal C}$, and 
$$(F,a)\in  \amalg_{F:X\sr U} D_p^{n-1}({\cal F})((X;F))=D_p^n({\cal F})(X)$$
Then one has
$$D_{\bf\Phi}^n(m)_X((F,a))=(\Phi(F)\circ \phi;D_{p'}^{n-1}({\cal F}')(\iota)(D_{\bf\Phi}^{n-1}(m)_{(X;F)}(a)))$$
where 
$$\iota=\iota_{\bf\Phi}^{X,F}:(\Phi(X);\Phi(F)\circ \phi)\sr \Phi((X;F))$$
is the morphism defined by (\ref{2015.04.08.eq1}) and (\ref{2015.04.08.eq2}).
\end{lemma}
\begin{proof}
We have
$$D_{\bf\Phi}^n(m)_X((F,a))=
$$$$\Phi D_{D_{p'}^{n-1}({\cal F}'),X}(D_p(D_{\bf\Phi}^{n-1}(m))_X((F,a)))=
\Phi D_{D_{p'}^{n-1}({\cal F}'),X}((F,D_{\bf\Phi}^{n-1}(m)_{(X;F)}(a)))=
$$$$(\Phi(F)\circ \phi;D_{p'}^{n-1}({\cal F}')(\iota)(D_{\bf\Phi}^{n-1}(m)_{(X;F)}(a)))$$
where the first equality is by definition of $D_{\bf\Phi}^n(m)$, the second by (\ref{2016.08.30.eq6}) and the third by (\ref{2016.12.16.eq4}). The lemma is proved.
\end{proof}
\begin{lemma}
\llabel{2016.12.18.l1}
In the context of Problem \ref{2016.12.18.prob1} consider a commutative square in $PreShv({\cal C})$ of the form
\begin{eq}
\llabel{2016.12.18.eq1}
\begin{CD}
{\cal F}_1 @>m_1>> \Phi^{\circ}({\cal F}_1')\\
@Vv VV @VV\Phi^{\circ}(v')V\\
{\cal F}_2 @>m_2>> \Phi^{\circ}({\cal F}_2')
\end{CD}
\end{eq}
Then, for any $n\in\nat$, the square
\begin{eq}
\llabel{2016.12.18.eq2}
\begin{CD}
D_p^n({\cal F}_1) @>D_{\bf\Phi}^n(m_1)>> \Phi^{\circ}(D_{p'}^n({\cal F}_1'))\\
@VD_p^n({v}) VV @VV\Phi^{\circ}(D_{p'}^n({v}'))V\\
D_p^n({\cal F}_2) @>D_{\bf\Phi}^n(m_2)>> \Phi^{\circ}(D_{p'}^n({\cal F}_2'))
\end{CD}
\end{eq}
commutes.
\end{lemma}
\begin{proof}
We proceed by induction on $n$. 

For $n=0$ the square (\ref{2016.12.18.eq2}) coincides with the square (\ref{2016.12.18.eq1}).

For the successor of $n$, (\ref{2016.12.18.eq2}) is the external square of the diagram
$$
\begin{CD}
D_p(D_p^n({\cal F}_1)) @>D_p(D_{\bf\Phi}^n(m_1))>> D_p(\Phi^{\circ}(D_{p'}^n({\cal F}_1'))) @>\Phi D_{D_{p'}^n({\cal F}_1')}>> \Phi^{\circ}(D_{p'}(D_{p'}^n({\cal F}_1')))\\
@VD_p(D_p^n({v})) VV @V D_p(\Phi^{\circ}(D_{p'}^n({v}')))VV @VV\Phi^{\circ}(D_{p'}(D_{p'}^n({v}')))V\\
D_p(D_p^n({\cal F}_2)) @>D_p(D_{\bf\Phi}^n(m_2))>> D_p(\Phi^{\circ}(D_{p'}^n({\cal F}_2'))) @>\Phi D_{D_{p'}^n({\cal F}_2')}>> \Phi^{\circ}(D_{p'}(D_{p'}^n({\cal F}_2')))
\end{CD}
$$
The left hand side square in this diagram is obtained by applying $D_p$ to the square (\ref{2016.12.18.eq2}) for $n$. It is commutative because $D_p$ is a functor and in particular satisfies the composition axiom (\ref{2016.12.18.eq4}). 

The right hand side square is commutative because $\Phi D$ is a natural transformation of functors that satisfies the axiom of compatibility with morphisms of presheaves. In our particular case this axiom is applied to the morphism of presheaves $D_{p'}^n(v')$. 

This completes the proof of the lemma.
\end{proof}
The following problem and construction are the only ones in this section where we change our context from considering a universe category functor to simply a functor between two categories.
\begin{problem}
\llabel{2016.12.18.prob3}
Given a functor $\Phi:{\cal C}\sr {\cal C}'$ between two categories to construct, for each $Y\in {\cal C}$, a morphism of presheaves
$$yo^{\Phi,Y}:Yo(Y)\sr \Phi^{\circ}(Yo(\Phi(Y)))$$
and to show that for a morphism $g:Y\sr Y'$ the square
\begin{eq}
\llabel{2016.12.18.eq8}
\begin{CD}
Yo(Y)@>yo^{\Phi,Y}>> \Phi^{\circ}(Yo(\Phi(Y)))\\
@VYo(g) VV @VV\Phi^{\circ}(Yo(\Phi(g))) V\\
Yo(Y')@>yo^{\Phi,Y'}>> \Phi^{\circ}(Yo(\Phi(Y')))
\end{CD}
\end{eq}
commutes.
\end{problem}
\begin{construction}\rm
\llabel{2016.12.18.constr3}
We need to define, for all $X\in {\cal C}$, functions
$$Yo(Y)(X)=Mor_{\cal C}(X,Y)\sr Mor_{{\cal C}'}(\Phi(X),\Phi(Y))=\Phi^{\circ}(Yo(\Phi(Y)))(X)$$
which we define as the restriction of $\Phi_{Mor}$ to $Mor_{\cal C}(X,Y)$:
\begin{eq}
\llabel{2016.12.18.eq7}
yo^{\Phi,Y}_X(f)=\Phi(f)
\end{eq}
Let us show that this family is a morphism of presheaves, i.e., that for any $a:X'\sr X$ the square
\begin{eq}
\llabel{2016.12.18.eq5}
\begin{CD}
Yo(Y)(X) @>yo^{\Phi,Y}_{X}>> \Phi^{\circ}(Yo(\Phi(Y)))(X)\\
@VYo(Y)(a)VV @VV\Phi^{\circ}(Yo(\Phi(Y)))(a)V\\
Yo(Y)(X') @>yo^{\Phi,Y}_{X'}>> \Phi^{\circ}(Yo(\Phi(Y)))(X')
\end{CD}
\end{eq}
commutes. Note that for an element $f':\Phi(X)\sr \Phi(Y)$ of $\Phi^{\circ}(Yo(\Phi(Y)))(X)$ we have
\begin{eq}
\llabel{2016.12.18.eq6}
\Phi^{\circ}(Yo(\Phi(Y)))(a)(f')=\Phi(a)\circ f'
\end{eq}
Let $f:X\sr Y$ be an element of $Yo(Y)(X)$.

Applying one path in (\ref{2016.12.18.eq5}) to $f$ we get
$$\Phi^{\circ}(Yo(\Phi(Y)))(a)(yo^{\Phi,Y}_{X}(f))=\Phi^{\circ}(Yo(\Phi(Y)))(a)(\Phi(f))=\Phi(a)\circ \Phi(f)$$
where the first equality is by (\ref{2016.12.18.eq7}) and the second is by (\ref{2016.12.18.eq6}). 

Applying another path we get
$$yo^{\Phi,Y}_{X'}(Yo(Y)(a)(f))=\Phi Yo(Y)_{X'}(a\circ f)=\Phi(a\circ f)$$
where the first equality is by definition of $Yo(Y)$ and the second by (\ref{2016.12.18.eq7}). 

We conclude that (\ref{2016.12.18.eq5}) commutes by the composition axiom of $\Phi$. 

Let $g:Y\sr Y'$ be a morphism. Note that for an element $f':\Phi(X)\sr \Phi(Y)$ of $\Phi^{\circ}(Yo(\Phi(Y)))(X)$ we have
\begin{eq}
\llabel{2016.12.18.eq9}
\Phi^{\circ}(Yo(\Phi(g)))(f')=f'\circ \Phi(g)
\end{eq}
Let us show that the square (\ref{2016.12.18.eq8}) commutes. Let $X\in {\cal C}$ and $f\in Yo(Y)(X)$. 

Applying one path in (\ref{2016.12.18.eq8}) to $f$ we get
$$\Phi^{\circ}(Yo(\Phi(g)))(yo^{\Phi,Y}(f))=\Phi^{\circ}(Yo(\Phi(g)))(\Phi(f))=\Phi(f)\circ \Phi(g)$$
where the first equality is by (\ref{2016.12.18.eq7}) and the second by (\ref{2016.12.18.eq9}).

Applying another path we get
$$yo^{\Phi,Y'}(Yo(g)(f))=yo^{\Phi,Y'}(f\circ g)=\Phi(f\circ g)$$
where the first equality is by the definition of $Yo(g)$ and the second by (\ref{2016.12.18.eq7}). We conclude that (\ref{2016.12.18.eq8}) commutes by the composition axiom of $\Phi$.

This completes the construction. 
\end{construction}
Recall that for $X,Y\in {\cal C}$ and $n\ge 0$ we have defined in (\ref{2016.12.24.eq1}) the set $D_p^n(X,Y)$ as follows:
$$D_p^n(X,Y)=D_p^n(Yo(Y))(X)$$
We also introduced before Remark \ref{2015.07.29.rem2} the  $\circ$-notation that we will use below. 
\begin{problem}
\llabel{2016.12.18.prob2}
With assumptions as above, to define, for all $X,Y\in {\cal C}$ and $n\ge 0$, functions
$${\bf\Phi}^n_{X,Y}:D_p^n(X,Y)\sr D_{p'}^n(\Phi(X),\Phi(Y))$$
\end{problem}
\begin{construction}\rm
\llabel{2016.12.18.constr2}
Applying Construction \ref{2016.12.18.constr1} to the morphism of presheaves $yo^{\Phi,Y}$ of Construction \ref{2016.12.18.constr3} we obtain morphisms of presheaves
$$D_{\bf\Phi}^n(yo^{\Phi,Y}):D_p^n(Yo(Y))\sr \Phi^{\circ}(D_{p'}^n(Yo(\Phi(Y))))$$
Evaluating this morphism on $X$ we obtain a function
\begin{eq}
\llabel{2016.12.20.eq3}
D_p^n(X,Y) = D_p^n(Yo(Y))(X)\sr \Phi^{\circ}(D_p^n(Yo(\Phi(Y))))(X) = D_{p'}^n(\Phi(X),\Phi(Y))
\end{eq}
\end{construction}
For $n=0$ we have 
$$D_p^0(X,Y)=Yo(Y)(X)=Mor_{\cal C}(X,Y)$$
and 
\begin{eq}
\llabel{2017.01.13.eq1}
{\bf\Phi}^0_{X,Y}=\Phi_{X,Y}
\end{eq}
that is, the restriction of $\Phi_{Mor}$ to the subset $Mor_{\cal C}(X,Y)$ of $Mor({\cal C})$. 

The explicit form of the function ${\bf\Phi}^n_{X,Y}$ when $n\ge 1$ is given by the following lemma.
\begin{lemma}
\llabel{2016.12.22.l2}
In the context of Problem \ref{2016.12.18.prob2}, let $n\ge 1$, $X,Y\in {\cal C}$ and
$$(F,a)\in \amalg_{F:X\sr U} D_p^{n-1}((X;F),Y)=D_p(D_p^{n-1}(Yo(Y)))(X)=D_p^n(X,Y)$$
Then one has
$${\bf\Phi}^n_{X,Y}((F,a))=(\Phi(F)\circ \phi, \iota\circ {\bf\Phi}^{n-1}_{(X;F),Y}(a))$$
where $\iota:(\Phi(X),\Phi(F)\circ\phi)\sr \Phi((X;F))$ is the morphism defined by (\ref{2015.04.08.eq1}) and (\ref{2015.04.08.eq2}).
\end{lemma}
\begin{proof}
By construction we have ${\bf\Phi}^n_{X,Y}=D_{\bf\Phi}^n(yo^{\Phi,Y})_X$. By Lemma \ref{2016.12.22.l1} we have
$$D_{\bf\Phi}^n(yo^{\Phi,Y})_X((F,a))=(\Phi(F)\circ \phi, D_{p'}^{n-1}(Yo(\Phi(Y)))(\iota)(D_{\bf\Phi}^{n-1}(yo^{\Phi,Y})_{(X;F)}(a)))$$
Again by construction we have ${\bf\Phi}^{n-1}_{(X;F),Y}=D_{\bf\Phi}^{n-1}(yo^{\Phi,Y})_{(X;F)}$ and $D_{p'}^{n-1}(Yo(\Phi(Y)))(\iota)=D_{p'}^{n-1}(\iota,Y)=\iota\circ -$. The lemma is proved.
\end{proof}
\begin{lemma}
\llabel{2016.12.20.l1}
In the context of Construction \ref{2016.12.18.constr2} one has:
\begin{enumerate}
\item let $f:X'\sr X$ be a morphism, then the square
\begin{eq}
\llabel{2016.12.20.eq1}
\begin{CD}
D_p^n(X,Y) @>{\bf\Phi}^n_{X,Y}>> D_{p'}^n(\Phi(X),\Phi(Y))\\
@VD_{p}^n(f,Y)VV @VVD_{p'}^n(\Phi(f),\Phi(Y))V\\
D_p^n(X',Y) @>{\bf\Phi}^n_{X',Y}>> D_{p'}^n(\Phi(X'),\Phi(Y))
\end{CD}
\end{eq}
commutes,
\item let $g:Y\sr Y'$ be a morphism, then the square
\begin{eq}
\llabel{2016.12.20.eq2}
\begin{CD}
D_p^n(X,Y) @>{\bf\Phi}^n_{X,Y}>> D_{p'}^n(\Phi(X),\Phi(Y))\\
@VD_{p}^n(X,g)VV @VVD_{p'}^n(\Phi(X),\Phi(g))V\\
D_p^n(X,Y') @>{\bf\Phi}^n_{X,Y'}>> D_{p'}^n(\Phi(X),\Phi(Y'))
\end{CD}
\end{eq}
commutes.
\end{enumerate}
\end{lemma}
\begin{proof}
Commutativity of (\ref{2016.12.20.eq1}) follows from (\ref{2016.12.20.eq3}) and the fact that $D_{\bf\Phi}^n(yo^{\Phi,Y})$ is a morphism of presheaves.

Commutativity of (\ref{2016.12.20.eq2}) follows from (\ref{2016.12.20.eq3}), the commutativity of (\ref{2016.12.18.eq8}) and Lemma \ref{2016.12.18.l1}.
\end{proof}
In the $\circ$-notation the assertion of Lemma \ref{2016.12.20.l1} looks as follows. Let $d\in D_p^n(X,Y)$. Then for $f:X'\sr X$ one has 
\begin{eq}
\llabel{2017.01.05.eq1}
\Phi(f)\circ {\bf\Phi}^n(d)={\bf\Phi}^n(f\circ d)
\end{eq}
and for $g:Y\sr Y'$ one has
\begin{eq}
\llabel{2017.01.05.eq2}
{\bf\Phi}^n(d)\circ \Phi(g)={\bf\Phi}^n(d\circ g)
\end{eq}

\subsection{Universe category functors and isomorphisms $u_n$ and $\wt{u}_n$}
\llabel{Sec.8}

By \cite[Construction 4.7]{Cfromauniverse} any universe category functor ${\bf \Phi}=(\Phi,\phi,\wt{\phi})$ from $({\cal C},p)$ to $({\cal C}',p)$ defines a homomorphism of C-systems
$$H:CC({\cal C},p)\sr CC({\cal C}',p')$$
Let $\psi_0:pt'\sr \Phi(pt)$ be the unique morphism. To define $H$ on objects, one uses the fact that
$$Ob(CC({\cal C},p))=\amalg_{n\ge 0} Ob_n({\cal C},p)$$
and defines $H(n,A)$ as $(n,H_n(A))$ where
$$H_n:Ob_n({\cal C},p)\sr Ob_n({\cal C}',p')$$
To obtain $H_n$ one defines by induction on $n$, pairs $(H_n,\psi_n)$ where $H_n$ is as above and $\psi_n$ is a family of isomorphisms 
$$\psi_{n}(A):int_n(H_n(A))\sr \Phi(int_n(A))$$
as follows:
\begin{enumerate}
\item for $n=0$, $H_0$ is the unique function from a one point set to a one point set and $\psi_{0}(A)=\psi_0$,
\item for the successor of $n$ one has 
\begin{eq}
\llabel{2016.12.10.eq1}
H_{n+1}(A,F)=(H_n(A),\psi_n(A)\circ\Phi(F)\circ \phi)
\end{eq}
and $\psi_{n+1}(A,F)$ is the unique morphism $int(H(A,F))\sr \Phi(int(A,F))$ such that
\begin{eq}
\llabel{2016.12.10.eq2}
\psi_{n+1}(A,F)\circ \Phi(Q(F))\circ\wt{\phi}=Q(\psi_n(A)\circ\Phi(F)\circ\phi)
\end{eq}
and
\begin{eq}
\llabel{2016.12.10.eq3}
\psi_{n+1}(A,F)\circ \Phi(p_{F})=p_{\psi_n(A)\circ\Phi(F)\circ \phi}\circ \psi_n(A)
\end{eq}
\end{enumerate}
The function $H:Ob(CC({\cal C},p))\sr Ob(CC({\cal C}',p'))$ is the sum of functions $H_n$. For $\Gamma=(n,A)$ in $Ob(CC({\cal C},p))$ we let $\psi(\Gamma)=\psi_n(A)$ so that $\psi$ is the sum of families $\psi_n$:
$$\psi(\Gamma):int(H(\Gamma))\sr \Phi(int(\Gamma))$$

The action of $H$ on morphisms is given by the condition that for $f:\Gamma'\sr\Gamma$, $H(f)$ is a unique morphism of the form $H(\Gamma')\sr H(\Gamma)$ such that
\begin{eq}
\llabel{2016.12.10.eq4}
int(H(f))=\psi(\Gamma')\circ\Phi(int(f))\circ\psi(\Gamma)^{-1}
\end{eq}
We will often write $H$ also for the functions $H_n$ and $\psi$ for the functions $\psi_n$. 
\begin{lemma}
\llabel{2016.12.20.l2}
The family of morphisms 
$$\psi(\Gamma):int(H(\Gamma))\sr \Phi(int(\Gamma))$$
is a natural isomorphism of functors 
$$\psi:H\circ int\sr int\circ \Phi$$
\end{lemma}
\begin{proof}
By construction, $\psi(\Gamma)$ is a family of morphisms of the form 
$(H\circ int)(\Gamma)\sr (int\circ \Phi)(\Gamma)$. It remains to verify that for $f:\Gamma'\sr \Gamma$ one has
$$\psi(\Gamma)\circ (int\circ \Phi)(f)=(H\circ int)(f)\circ \psi(\Gamma')$$
This equality is equivalent to (\ref{2016.12.10.eq4}). 
\end{proof}
The following lemma describes the interaction between the isomorphisms $u_1$ and $\wt{u}_1$ and universe category functors. 
\begin{lemma}
\llabel{2015.03.21.l4}
Let $(\Phi,\phi,\wt{\phi})$ be universe category functor. Then:
\begin{enumerate}
\item for $T\in Ob_1(\Gamma)$ one has 
$$u_{1,H(\Gamma)}(H(T))=\psi(\Gamma)\circ\Phi(u_{1,\Gamma}(T))\circ\phi$$
\item for $o\in \wOb_1(\Gamma)$ one has
$$\wt{u}_{1,H(\Gamma)}(H(o))=\psi(\Gamma)\circ\Phi(\wt{u}_{1,\Gamma}(o))\circ\wt{\phi}$$ 
\end{enumerate}
\end{lemma}
\begin{proof}
Let $\Gamma=(n,A)$.

In the case of $T\in Ob_1(\Gamma)$, if $T=(n+1,(A,F))$ then
$$u_1(H(T))=u_1(n+1,H(A,F))=u_1(n+1,(H(A), \psi(\Gamma)\circ\Phi(F)\circ\phi))=\psi(\Gamma)\circ\Phi(F)\circ\phi$$
where the last equality is by (\ref{2015.04.30.eq3a}). 

In the case of $o\in \wOb_1(\Gamma)$, if $\partial(o)=(n+1,(A,F))$ then $\partial(H(o))=(n+1,H(A,F))$. Since $o:\Gamma\sr \partial(o)$ we have 
\begin{eq}
\llabel{2016.12.10.eq6}
H(o)=\psi(\Gamma)\circ \Phi(int(o))\circ \psi(A,F)^{-1}
\end{eq}
and 
$$\wt{u}_1(H(o))=
H(o)\circ Q(u_1(n+1,H(A,F)))=
H(o)\circ Q(\psi(A)\circ \Phi(F)\circ \phi)=
$$$$H(o)\circ \psi(A,F)\circ \Phi(Q(F))\circ \wt{\phi}=
\psi(\Gamma)\circ \Phi(int(o))\circ \psi(A,F)^{-1}\circ \psi(A,F)\circ \Phi(Q(F))\circ \wt{\phi}=
$$$$\psi(\Gamma)\circ \Phi(int(o))\circ \Phi(Q(F))\circ \wt{\phi}=
\psi(\Gamma)\circ \Phi(int(o)\circ Q(F))\circ \wt{\phi}=
$$$$\psi(\Gamma)\circ \Phi(\wt{u}_1(o))\circ \wt{\phi}
$$
where the first equality is by (\ref{2015.04.30.eq4a}), the second by (\ref{2016.12.10.eq1}) and (\ref{2015.04.30.eq3a}), the third by (\ref{2016.12.10.eq2}), the fourth by (\ref{2016.12.10.eq6}) and the seventh again by (\ref{2015.04.30.eq4a}).
\end{proof}
We now want to express the assertion of Lemma \ref{2015.03.21.l4} in terms of the commutativity of a diagram of natural transformations of presheaves on $CC({\cal C},p)$.

We will use the natural transformation $\psi^{\circ}$ that $\psi$ defines on the corresponding functors between the categories of presheaves. Note that for a natural transformation $a:\Phi_1\sr \Phi_2$ of functors of the form ${\cal C}\sr {\cal C}'$ and a presheaf $F'$ on ${\cal C}'$ we have
$$\Phi_2^{\circ}(F')(X)=F'(\Phi_2(X))\sr F'(\Phi_1(X))=\Phi_1^{\circ}(F')(X)$$
that is, for $a:\Phi_1\sr \Phi_2$ we have $a^{\circ}:\Phi_2^{\circ}\sr \Phi_1^{\circ}$. In particular, in the case of $\psi$ we have:
$$\psi^{\circ}:\Phi^{\circ}\circ int^{\circ}=(int\circ \Phi)^{\circ}\sr (H\circ int)^{\circ}=int^{\circ}\circ H^{\circ}$$
\begin{lemma}
\llabel{2016.12.20.l3}
In the context of Lemma \ref{2015.03.21.l4} the following two diagrams of natural transformations of presheaves on $CC({\cal C},p)$ commute:
$$
\begin{CD}
\Ob_1 @>u_1>> int^{\circ}(Yo(U))\\
@. @VV int^{\circ}(yo^{\Phi,U}) V\\
@.  int^{\circ}(\Phi^{\circ}(Yo(\Phi(U))))\\
@VVH\Ob_1 V @VVint^{\circ}(\Phi^{\circ}(Yo(\phi))) V\\
@. int^{\circ}(\Phi^{\circ}(Yo(U')))\\
@. @VV\psi^{\circ}_{Yo(U')} V\\
H^{\circ}(\Ob_1) @>H^{\circ}(u_1)>> H^{\circ}(int^{\circ}(Yo(U'))
\end{CD}
\spc\spc\spc\spc
\begin{CD}
\wOb_1 @>\wt{u}_1>> int^{\circ}(Yo(\wt{U}))\\
@. @VV int^{\circ}(yo^{\Phi,\wt{U}}) V\\
@.  int^{\circ}(\Phi^{\circ}(Yo(\Phi(\wt{U}))))\\
@VVH\wOb_1 V @VVint^{\circ}(\Phi^{\circ}(Yo(\wt{\phi}))) V\\
@. int^{\circ}(\Phi^{\circ}(Yo(\wt{U}')))\\
@. @VV\psi^{\circ}_{Yo(\wt{U}')} V\\
H^{\circ}(\wOb_1) @>H^{\circ}(\wt{u}_1)>> H^{\circ}(int^{\circ}(Yo(\wt{U}'))
\end{CD}
$$
\end{lemma}
\begin{proof}
Consider the first diagram. For $\Gamma$ and $T\in \Ob_1(\Gamma)$ one path in the diagram applied to $T$ gives us
$$\psi^{\circ}(int^{\circ}(\Phi^{\circ}(Yo(\phi)))(int^{\circ}(yo^{\Phi,U}(u_{1,\Gamma}(T)))))=
\psi^{\circ}(int^{\circ}(\Phi^{\circ}(Yo(\phi)))(\Phi(u_{1,\Gamma}(T))))=
$$$$\psi^{\circ}(\Phi(u_{1,\Gamma}(T))\circ \phi)=
\psi(\Gamma)\circ \Phi(u_{1,\Gamma}(T))\circ \phi$$
while the other path gives
$$H^{\circ}(u_1)(H\Ob_1(T))=u_{1,H(\Gamma)}(H(T))$$
The equality of these two expressions is the statement of Lemma \ref{2015.03.21.l4}(1).

The case of the second diagram is strictly parallel. The lemma is proved.
\end{proof}

Consider now isomorphisms $u_{n}$ and $\wt{u}_n$ for general $n\ge 1$. 
\begin{lemma}
\llabel{2016.12.20.l4}
Let ${\bf\Phi}=(\Phi,\phi,\wt{\phi})$ be a universe category functor and $n\ge 1$. Then
\begin{enumerate}
\item for $T\in \Ob_n(\Gamma)$ one has
\begin{eq}
\llabel{2016.12.20.eq5}
u_{n,H(\Gamma)}(H(T))=\psi(\Gamma)\circ ({\bf\Phi}^{n-1}_{int(\Gamma),U}(u_{n,\Gamma}(T))\circ \phi)
\end{eq}
\item for $o\in \wOb_n(\Gamma)$ one has
\begin{eq}
\llabel{2016.12.20.eq6}
\wt{u}_{n,H(\Gamma)}(H(o))=\psi(\Gamma)\circ ({\bf\Phi}^{n-1}_{int(\Gamma),\wt{U}}(\wt{u}_{n,\Gamma}(o))\circ \wt{\phi})
\end{eq}
\end{enumerate}
\end{lemma}
\begin{proof}
Let us verify first that the right hand side of (\ref{2016.12.20.eq5}) is defined and belongs to the same set as the left hand side. 

By (\ref{2016.11.12.eq2}) and (\ref{2016.11.22.eq1}), we have $u_{n,\Gamma}(T)\in D_p^{n-1}(int(\Gamma),U)$. Therefore, 
$${\bf\Phi}^{n-1}_{int(\Gamma),U}(u_{n,\Gamma}(T))\in D_{p'}^{n-1}(\Phi(int(\Gamma)),\Phi(U))$$
Since $\phi:\Phi(U)\sr U'$ and $\psi(\Gamma):int(H(\Gamma))\sr \Phi(int(\Gamma))$ we have
$$\psi(\Gamma)\circ ({\bf\Phi}^{n-1}_{int(\Gamma),U}(u_{n,\Gamma}(T))\circ \phi)\in D_{p'}^{n-1}(int(H(\Gamma)),U')$$
on the other hand $u_{n,H(\Gamma)}(H(T))$ is an element of $D_{p'}^{n-1}(int(H(\Gamma)),U')$ as well. Therefore, (\ref{2016.12.20.eq5}) is an equality between two elements of the same set.

To prove (\ref{2016.12.20.eq5}) we proceed by induction on $n$. 
For $n=1$ this equality is the same as the equality of Lemma \ref{2015.03.21.l4}(1). 

For the successor of $n\ge 1$ we reason as follows. Let $T'=ft^n(T)\in\Ob_1(\Gamma)$ and let us abbreviate $u_{i,-}$ to $u_n$. By (\ref{2016.12.22.eq1}) and since $H$ commutes with $ft$ we have
\begin{eq}
\llabel{2016.12.24.eq4}
u_{n+1}(H(T))=(u_{1}(ft^n(H(T))),u_{n}(H(T)))=(u_{1}(H(T')),u_{n}(H(T)))
\end{eq}
By the inductive assumption we have
\begin{eq}
\llabel{2016.12.24.eq5}
u_{n}(H(T))=
\psi(T')\circ ({\bf\Phi}^{n-1}_{int(T'),U}(u_{n}(T))\circ \phi)
\end{eq}
On the other hand, by (\ref{2016.12.22.eq1}), Lemma \ref{2016.12.22.l2} and (\ref{2015.05.02.eq1a}) we have
$${\bf\Phi}^n_{int(\Gamma),U}(u_{n+1}(T))=
{\bf\Phi}^n_{int(\Gamma),U}((u_{1}(T'),u_{n}(T)))=
$$$$(\Phi(u_{1}(T'))\circ \phi,\iota\circ {\bf\Phi}^{n-1}_{(int(\Gamma);u_{1}(T')),U}(u_{n}(T)))=(\Phi(u_{1}(T'))\circ \phi,\iota\circ {\bf\Phi}^{n-1}_{int(T'),U}(u_{n}(T)))$$
where 
$$\iota=\iota_{\bf\Phi}^{int(\Gamma),u_{1}(T')}:(\Phi(int(\Gamma));\Phi(u_{1}(T'))\circ \phi)\sr \Phi((int(\Gamma);u_{1}(T')))=\Phi(int(T'))$$
is defined by the obvious analogs of (\ref{2015.04.08.eq1}) and (\ref{2015.04.08.eq2}). 

By Lemma \ref{2016.12.24.l1}(2) we have
$$(\Phi(u_{1}(T'))\circ \phi,\iota\circ ({\bf\Phi}^{n-1}_{int(T'),U}(u_{n}(T))))\circ \phi=(\Phi(u_{1}(T'))\circ \phi,(\iota\circ {\bf\Phi}^{n-1}_{int(T'),U}(u_{n}(T)))\circ \phi)$$
Next, by Lemma \ref{2016.12.24.l1}(1) we have
$$\psi(\Gamma)\circ (\Phi(u_{1}(T'))\circ \phi,(\iota\circ {\bf\Phi}^{n-1}_{int(T'),U}(u_{n}(T)))\circ \phi)=
$$$$(\psi(\Gamma)\circ \Phi(u_1(T'))\circ \phi, Q(\psi(\Gamma),\Phi(u_1(T'))\circ\phi)\circ 
((\iota\circ {\bf\Phi}^{n-1}_{int(T'),U}(u_{n}(T)))\circ \phi))$$
It remains to compare the last expression with (\ref{2016.12.24.eq4}). Both expressions are pairs. The first components of these pairs are equal by Lemma \ref{2015.03.21.l4}(1). To show that the second components are equal we need, in view of (\ref{2016.12.24.eq5}), to show that 
$$Q(\psi(\Gamma),\Phi(u_1(T'))\circ\phi)\circ 
((\iota\circ {\bf\Phi}^{n-1}_{int(T'),U}(u_{n}(T)))\circ \phi)=\psi(T')\circ ({\bf\Phi}^{n-1}_{int(T'),U}(u_{n}(T))\circ \phi)$$
In view of the ``associativities'' of Lemma \ref{2017.01.07.l1} it is sufficient to show that
\begin{eq}
\llabel{2016.12.24.eq7}
Q(\psi(\Gamma),\Phi(u_1(T'))\circ\phi)\circ\iota=\psi(T')
\end{eq}
where
$$\iota=\iota_{\bf\Phi}^{int(\Gamma),u_1(T')}:(\Phi(int(\Gamma));\Phi(u_1(T'))\circ\phi)\sr \Phi((int(\Gamma);u_1(T')))$$
Let $\Gamma=(m,A)$ and $F=u_{1,\Gamma}(T')$. Then $T'=(m+1,(A,F))$ and $\psi(T')=\psi((A,F))$ is the unique morphism that satisfies the equations (\ref{2016.12.10.eq2}) and (\ref{2016.12.10.eq3}). Therefore, to prove (\ref{2016.12.24.eq7}) we need to show that the following equalities hold:
\begin{eq}
\llabel{2016.12.24.eq8}
Q(\psi(A),\Phi(F)\circ\phi)\circ \iota_{\bf\Phi}^{int(A),F}\circ \Phi(Q(F))\circ \wt{\phi}=Q(\psi(A)\circ \Phi(F)\circ \phi)
\end{eq}
\begin{eq}
\llabel{2016.12.24.eq9}
Q(\psi(A),\Phi(F)\circ\phi)\circ \iota_{\bf\Phi}^{int(A),F}\circ \Phi(p_F)=p_{\psi(A)\circ \Phi(F)\circ \phi}\circ \psi(A)
\end{eq}
For (\ref{2016.12.24.eq8}) we have
$$Q(\psi(A),\Phi(F)\circ\phi)\circ \iota\circ \Phi(Q(F))\circ \wt{\phi}=
Q(\psi(A),\Phi(F)\circ\phi)\circ Q(\Phi(F)\circ \phi)=
Q(\psi(A)\circ \Phi(F)\circ\phi)$$
where the first equality is by (\ref{2015.04.08.eq2}) and the second one by (\ref{2016.08.24.eq4}).

For (\ref{2016.12.24.eq9}) we have
$$Q(\psi(A),\Phi(F)\circ\phi)\circ \iota\circ \Phi(p_F)=
Q(\psi(A),\Phi(F)\circ\phi)\circ p_{\Phi(F)\circ \phi}=p_{\psi(A)\circ \Phi(F)\circ\phi}\circ \psi(A)$$
where the first equality is by  (\ref{2015.04.08.eq1}) and the second one by (\ref{2016.12.02.eq4}) and (\ref{2016.11.10.eq1a}). 

A strictly parallel reasoning applies to the proof of (\ref{2016.12.20.eq6}). 

This completes the proof of Lemma \ref{2016.12.20.l4}.
\end{proof}

From (\ref{2016.12.20.eq5}), using the fact that $u_n$ is a bijection, for $d\in D_p^{n-1}(int(\Gamma),U)$, we have
\begin{eq}
\llabel{2017.01.13.eq6}
H(u_n^{-1}(d))=u_n^{-1}(\psi(\Gamma)\circ ({\bf\Phi}^{n-1}(d)\circ \phi))
\end{eq}
and from (\ref{2016.12.20.eq6}), using the fact that $\wt{u}_n$ is a bijection, for $d\in D_p^{n-1}(int(\Gamma),\wt{U})$, we have
\begin{eq}
\llabel{2017.01.13.eq7}
H(\wt{u}_n^{-1}(d))=\wt{u}_n^{-1}(\psi(\Gamma)\circ ({\bf\Phi}^{n-1}(d)\circ \wt{\phi}))
\end{eq}

\subsection{Universe category functors and the $I_p$ construcion}
\llabel{Sec.9}

Let $({\cal C},p)$ and $({\cal C}',p')$ be locally cartesian closed universe categories with binary product structure as considered in Section \ref{Sec.6}. Let ${\bf\Phi}:({\cal C},p)\sr ({\cal C}',p')$ be a universe category functor. No assumption is made about the compatibility of $\Phi$ with the locally cartesian closed or binary product structures.

In what follows we omit the indexes at $\eta_n$, $\eta^!_n$ and ${\bf\Phi}^n$ where no confusion is possible. 
\begin{problem}
\llabel{2015.03.21.prob1}
In the context introduced above to construct, for any $n\ge 0$ and $Y\in {\cal C}$, a morphism 
$$\chi_{{\bf\Phi},n}(Y):\Phi(I^n_p(Y))\sr I^n_{p'}(\Phi(Y))$$
such that for any $g:Y\sr Y'$ the square
\begin{eq}
\llabel{2016.12.30.eq1}
\begin{CD}
\Phi(I^n_p(Y)) @>\chi_{{\bf\Phi},n}(Y)>> I^n_{p'}(\Phi(Y))\\
@V\Phi(I^n_p(g))VV @VVI^n_{p'}(\Phi(g))V\\
\Phi(I^n_p(Y')) @>\chi_{{\bf\Phi},n}(Y')>> I^n_{p'}(\Phi(Y'))
\end{CD}
\end{eq}
commutes.
\end{problem}
\begin{construction}\rm
\llabel{2015.03.21.constr1}
We set
$$\chi_{{\bf\Phi},n}(Y)=\eta_n({\bf\Phi}^n(Id^n_Y))$$
where $Id^n_Y$ is defined in (\ref{2017.01.07.eq1}). In what follows we often omit the index $\bf\Phi$ at $\chi$. Let $g:Y\sr Y'$. Let us show that the square (\ref{2016.12.30.eq1}) commutes. We have
$$\chi_n(Y)\circ I_{p'}^n(\Phi(g))=
$$$$\eta_n({\bf\Phi}^n(Id^n_Y))\circ I_{p'}^n(\Phi(g))=
\eta_n({\bf\Phi}^n(Id^n_Y)\circ \Phi(g))=
\eta_n({\bf\Phi}^n(Id^n_Y\circ g))=
$$$$\eta_n({\bf\Phi}^n(\eta_n^!(I_p^n(g))))$$
where the first equality is by the definition of $\chi_n$, the second by (\ref{2017.01.03.eq3}), the third by (\ref{2017.01.05.eq2}), and the fourth by (\ref{2017.01.07.eq4}). 

On the other hand we have, 
$$\Phi(I_p^n(g))\circ \chi_n(Y')=
$$$$\Phi(I_p^n(g))\circ \eta_n({\bf\Phi}^n(Id^n_{Y'}))=
\eta_n(\Phi(I_p^n(g))\circ {\bf\Phi}^n(Id^n_{Y'}))=
\eta_n({\bf\Phi}^n(I_p^n(g)\circ Id^n_{Y'}))=
$$$$\eta_n({\bf\Phi}^n(\eta^!_n(I_p^n(g))))$$
where the first equality is by the definition of $\chi_n$, the second by (\ref{2017.01.03.eq2}), the third by (\ref{2017.01.05.eq1}), and the fourth by (\ref{2017.01.07.eq3}).

This shows that the square (\ref{2016.12.30.eq1}) commutes and completes the construction. 
\end{construction}
We have
\begin{eq}
\llabel{2017.01.13.eq3}
\chi_0(Y)=Id_{\Phi(Y)}
\end{eq}
Indeed,
$$\chi_0(Y)=\eta_0(\Phi^0(Id^n_Y))=\eta_0(\Phi^0(\eta_0^!(Id_{I_p^0(Y)})))=\Phi^0(Id_Y)=\Phi(Id_Y)=Id_{\Phi(Y)}$$
where the first equality is by the definition of $\xi_n$ in Construction \ref{2015.03.21.constr1}, the second by the definition of $Id^n_Y$ in (\ref{2017.01.07.eq1}), the third by the fact that $\eta_0=Id$ by the definition of $\eta_n$ in Construction \ref{2016.12.02.constr1}, the fourth by  (\ref{2017.01.13.eq1}) and the fifth by the identity axiom of the functor $\Phi$. 

We will also use the following formula.
\begin{lemma}
\llabel{2017.01.07.l3}
In the notation introduced above and $d\in D_p^n(X,Y)$ one has
$$\eta_n({\bf\Phi}^n(d))=\Phi(\eta_n(d))\circ\chi_n(Y)$$
\end{lemma}
\begin{proof}
We have
$$\Phi(\eta_n(d))\circ\chi_n(Y)=
$$$$\Phi(\eta_n(d))\circ\eta_n({\bf\Phi}^n(Id^n_Y))=
\eta_n(\Phi(\eta_n(d))\circ {\bf\Phi}^n(Id^n_Y))=
\eta_n({\bf\Phi}^n(\eta_n(d)\circ Id^n_Y))=
$$$$\eta_n({\bf\Phi}^n(d))$$
where the first equality is by the definition of $\chi_n$, the second by (\ref{2017.01.03.eq2}), the third by (\ref{2017.01.05.eq1}), and the fourth by (\ref{2017.01.07.eq5}).
\end{proof}

\subsection{Universe category functors and the isomorphisms $\mu_n$ and $\wt{\mu}_n$}
\llabel{Sec.10}
For a universe category functor $(\Phi,\phi,\wt{\phi})$ and $n\ge 0$ let us denote by 
$$\xi_{{\bf\Phi},n}:\Phi(I_p^n(U))\sr I_{p'}^n(U')$$
the composition $\chi_{{\bf\Phi},n}(U)\circ I_{p'}^n(\phi)$ and by
$$\wt{\xi}_{{\bf\Phi},n}:\Phi(I_p^n(\wt{U}))\sr I_{p'}^n(\wt{U}')$$
the composition $\chi_{{\bf\Phi},n}(\wt{U})\circ I_{p'}^n(\wt{\phi})$. 

From (\ref{2017.01.13.eq3}) we have
\begin{eq}
\llabel{2017.01.13.eq5}
\xi_0=\phi\spc\spc
\wt{\xi}_0=\wt{\phi}
\end{eq}

In view of the commutativity of the squares (\ref{2016.12.30.eq1}) and (\ref{2015.03.21.sq1}) and the composition axiom for the functor $I_{p'}^n$ the squares
\begin{eq}
\llabel{2017.01.01.eq6}
\begin{CD}
\Phi(I_p^n(\wt{U})) @>\wt{\xi}_{{\bf\Phi},n}>> I_{p'}^n(\wt{U}')\\
@V\Phi(I_p^n(\partial))VV @VVI_{p'}^n(\partial)V\\
\Phi(I_p^n(U)) @>\xi_{{\bf\Phi},n}>> I_{p'}^n(U')
\end{CD}
\end{eq}
commute.
\begin{lemma}
\llabel{2015.05.06.l2}
Let $(\Phi,\phi,\wt{\phi})$ be a universe category functor, $\Gamma\in Ob(CC({\cal C},p))$ and $n\ge 1$. Then one has:
\begin{enumerate}
\item for $T\in Ob_n(\Gamma)$
\begin{eq}
\llabel{2017.01.01.eq7}
\mu_{n,H(\Gamma)}(H(T))=\psi(\Gamma)\circ \Phi(\mu_{n,\Gamma}(T))\circ \xi_{{\bf\Phi},n-1}
\end{eq}
\item for $o\in \wOb_n(\Gamma)$
\begin{eq}
\llabel{2017.01.01.eq8}
\wt{\mu}_{n,H(\Gamma)}(H(o))=\psi(\Gamma)\circ \Phi(\wt{\mu}_{n,\Gamma}(o))\circ \wt{\xi}_{{\bf\Phi},n-1}
\end{eq}
\end{enumerate}
\end{lemma}
\begin{proof}
Let us show first that the right hand sides of (\ref{2017.01.01.eq7}) and (\ref{2017.01.01.eq8}) are defined and belong to the same sets as the left hand sides.

Indeed, by (\ref{2017.01.03.eq4}), $\mu_{n,\Gamma}(T)$ is an element of $Mor_{\cal C}(int(\Gamma),I_p^{n-1}(U))$ and therefore $\Phi(\mu_{n,\Gamma}(T))$ is an element of $Mor_{\cal C'}(\Phi(int(\Gamma)),\Phi(I_p^{n-1}(U)))$.

The morphism $\psi(\Gamma)$ is of the form $int(H(\Gamma))\sr \Phi(int(\Gamma))$ and the morphism $\xi_{{\bf\Phi},n-1}$ is of the form $\Phi(I_p^{n-1}(U))\sr I_{p'}^{n-1}(U')$. Therefore, the composition on the right hand side of (\ref{2017.01.01.eq7}) is defined and belongs to the same set $Mor_{\cal C'}(int(H(\Gamma)),I_{p'}^{n-1}(U'))$ as $\mu_{n,H(\Gamma)}(H(T))$. 

A parallel reasoning shows that the right hand side of (\ref{2017.01.01.eq8}) is defined and both sides are elements of the set $Mor_{\cal C'}(int(H(\Gamma)),I_{p'}^{n-1}(\wt{U}'))$.

Next, we have
$$\mu_{n,H(\Gamma)}(H(T))=
int^{\circ}(\eta_{n-1,U'})_{H(\Gamma)}(u_n(H(T)))=
\eta_{n-1,U',int(H(\Gamma))}(u_n(H(T)))=
$$$$\eta_{n-1,U',int(H(\Gamma))}(\psi(\Gamma)\circ ({\bf\Phi}^{n-1}(u_n(T))\circ\phi))=
\psi(\Gamma)\circ \eta_{n-1,U',\Phi(int(\Gamma))}({\bf\Phi}^{n-1}(u_n(T))\circ\phi)=
$$$$\psi(\Gamma)\circ \eta_{n-1,\Phi(U),\Phi(int(\Gamma))}({\bf\Phi}^{n-1}(u_n(T)))\circ I_{p'}^{n-1}(\phi)$$
where the first equality is by the definition of $\mu_n$ (cf. Construction \ref{2015.03.17.constr2}), the second by the definition of $int^{\circ}$, the third by (\ref{2016.12.20.eq5}), the fourth by (\ref{2017.01.03.eq2}) and the fifth by (\ref{2017.01.03.eq3}).

Next
$$\eta_{n-1}({\bf \Phi}^{n-1}(u_n(T)))\circ I_{p'}^{n-1}(\phi)=
$$$$\Phi(\eta_{n-1}(u_n(T)))\circ\chi_{n-1}(U)\circ I_{p'}^{n-1}(\phi)=
\Phi(\eta_{n-1}(u_n(T)))\circ \xi_{n-1}=
$$$$\Phi(\mu_n(T))\circ \xi_{n-1}$$
where the first equality holds by Lemma \ref{2017.01.07.l3}, the second one by the definition of $\xi_{n}$ and the third one by the definition of $\mu_n$. This reasoning proves (\ref{2017.01.01.eq7}).

The proof of (\ref{2017.01.01.eq8}) is strictly parallel to the proof of (\ref{2017.01.01.eq7}).

The lemma is proved. 
\end{proof}

From (\ref{2017.01.01.eq7}), using the fact that $\mu_n$ is a bijection, for $F\in Mor_{\cal C}(int(\Gamma),I_p^{n-1}(U))$, we have
\begin{eq}
\llabel{2017.01.13.eq4}
H(\mu_n^{-1}(F))=\mu_n^{-1}(\psi(\Gamma)\circ \Phi(F)\circ \xi_{n-1})
\end{eq}
and from (\ref{2017.01.01.eq8}), using the fact that $\wt{\mu}_n$ is a bijection, for $F\in Mor_{\cal C}(int(\Gamma), I_p^{n-1}(\wt{U}))$, we have
\begin{eq}
\llabel{2017.01.13.eq2}
H(\wt{\mu}_n^{-1}(F))=\wt{\mu}_n^{-1}(\psi(\Gamma)\circ \Phi(F)\circ \wt{\xi}_{n-1})
\end{eq}

\section{Appendices}
The facts discussed and proved in the following appendices are certainly well known. We had to repeat them here because we need to fix notations and because there is a number of facts whose proofs I could not find in the literature.
\subsection{Categories with binary products and binary cartesian closed categories} 
\label{App.1}
Let $\cal C$ be a category.
\begin{definition}
\llabel{2016.12.02.def1}
A binary product diagram is a pair of morphisms of the form $(pr_1:bp\sr X,pr_2:bp\sr Y)$ such that for all $A\in {\cal C}$ the function 
\begin{eq}
\llabel{2016.12.02.eq2a}
Mor_{\cal C}(A,bp)\sr Mor_{\cal C}(A,X)\times Mor_{\cal C}(A,Y)
\end{eq}
given by $a\mapsto (a\circ pr_1, a\circ pr_2)$ is a bijection.

The structure of binary products on $\cal C$ is a family,  parametrized by pairs of objects $(X,Y)\in {\cal C}\times {\cal C}$, of binary product diagrams $(pr_1(X,Y):bp(X,Y)\sr X,pr_2(X,Y):bp(X,Y)\sr Y)$. 
\end{definition}
Unless another notation is given, as for the binary products in the slice categories considered below, the object $bp(X,Y)$ is denoted by $X\times Y$ and the structural morphisms from $X\times Y$ to $X$ and $Y$ by $pr_1^{X,Y}$ and $pr_2^{X,Y}$ respectively. We will often abbreviate the notation $pr_i^{X,Y}$ to $pr_i$. 

The following lemma expresses the well know ``uniqueness'' property of the binary products. We need its explicit form because in the next lemma we will need to state and prove that the corresponding ``canonical'' isomorphisms are natural. 
\begin{lemma}
\llabel{2016.12.02.l1}
Let $(pr_{1,i}:bp_i\sr X,pr_{2,i}:bp_i\sr Y)$, where $i=1,2$, be two binary product diagrams.  Let $\iota_{1,2}:bp_1\sr bp_2$ be the morphism such that $\iota_{1,2}\circ pr_{1,2}=pr_{1,1}$ and $\iota_{1,2}\circ pr_{2,2}=pr_{2,1}$ and $\iota_{2,1}:bp_2\sr bp_1$ be the morphism given by the symmetric condition. Then $\iota_{1,2}$ and $\iota_{2,1}$ are mutually inverse isomorphisms.
\end{lemma}
\begin{proof}
To show that $\iota_{1,2}\circ \iota_{2,1}=Id_{bp_1}$ we need to compare two morphisms whose codomain is a binary product. To do it it is sufficient, because of the injectivity of (\ref{2016.12.02.eq2a}), to prove that their compositions with the two projections are equal. This follows by simple rewriting. The same applies to the second composition.
\end{proof}
\begin{lemma}
\llabel{2015.04.16.l1}
Let $\cal C$ be a category. Consider four binary product diagrams  
$(pr_{1,i}:bp_i\sr X,pr_{2,i}:bp_i\sr Y)$ and $(pr'_{1,i}:bp'_i\sr X',pr'_{2,i}:bp'_i\sr Y')$ where $i=1,2$. Let $\iota=\iota_{1,2}:pb_1\sr pb_2$ be as in Lemma \ref{2016.12.02.l1} and similarly $\iota':pb_1'\sr pb_2'$. Let $a:X'\sr X$ and $b:Y'\sr Y$.

Let $c_i(a,b):pb_i'\sr pb_i$ be the unique morphisms such that $c_i(a,b)\circ pr_{1,i}=pr'_{1,i}\circ a$ and $c_i(a,b)\circ pr_{2,i}=b\circ pr'_{2,i}$. Then the square
$$
\begin{CD}
pb_1' @>c_1(a,b)>> pb_1\\
@V\iota' VV @VV\iota V\\
pb_2' @>c_2(a,b)>> pb_2
\end{CD}
$$
commutes, i.e., $c_1(a,b)\circ \iota=\iota'\circ c_2(a,b)$.
\end{lemma}
\begin{proof}
Since $(pr_{1,2},pr_{2,2})$ is a binary product digagram it is sufficient to prove that
$$c_1(a,b)\circ \iota\circ pr_{1,2}=\iota'\circ c_2(a,b)\circ pr_{1,2}$$
and
$$c_1(a,b)\circ \iota\circ pr_{2,2}=\iota'\circ c_2(a,b)\circ pr_{2,2}$$
For the first one we have:
$$c_1(a,b)\circ \iota\circ pr_{1,2}=c_1(a,b)\circ pr_{1,1}=pr'_{1,1}\circ a$$
and
$$\iota'\circ c_2(a,b)\circ pr_{1,2}=\iota'\circ pr'_{1,2}\circ a=pr'_{1,1}\circ a$$
The verification of the second equality is similar.
\end{proof}

Given a category with binary products and morphisms $a:X\sr X'$, $b:Y\sr Y'$ denote by $a\times b:X\times Y\sr X'\times Y'$ the unique morphism such that $(a\times b)\circ pr_1=pr_1\circ a$ and $(a\times b)\circ pr_2=pr_2\circ b$. 

One has
\begin{eq}
\llabel{2016.11.26.eq1}
Id_{X\times Y}=Id_X\times Id_Y
\end{eq}
and for $a,b$ as above and $a':X'\sr X''$, $b':X'\sr X"$ one has
\begin{eq}
\llabel{2016.11.26.eq2}
(a\times b)\circ (a'\times b')=(a\circ a')\times (b\circ b')
\end{eq}
One proves these two equalities by composing both sides with $pr_1$ and $pr_2$ and using the uniqueness part of the binary product axiom.

From (\ref{2016.11.26.eq2}) one derives
\begin{eq}
\llabel{2016.11.28.eq3}
(a\times Id_Y)\circ (a'\times Id_Y)=(a\circ a')\times Id_Y
\end{eq}
and
\begin{eq}
\llabel{2016.11.28.eq4}
(Id_X\times b)\circ (Id_X\times b')=Id_X\times (b\circ b')
\end{eq}

The definition of a binary cartesian closed structure given below differs slightly from the definition of the cartesian closed structure given in \cite[IV.6]{MacLane} in that, that we do not require the specification of a final object  but only of binary products. The rest of the definition is identical to the one in \cite{MacLane}, but written more explicitly in order to introduce the notations that are used in proofs in the main part of the paper. 

Since we never use the definition of \cite[IV.6]{MacLane} we will often write ``cartesian closed'' instead of ``binary cartesian closed'', though we are not assuming a final object.
\begin{definition}
\llabel{2016.11.28.def1}
The (binary) cartesian closed structure on a category $\cal C$ is a collection of data of the form:
\begin{enumerate}
\item the structure of a category with binary products on $\cal C$,
\item for all $X,Y\in {\cal C}$ an object $\uu{Hom}(X,Y)$,
\item for all $X$ and $b:Y\sr Y'$ a morphism
$$\uu{Hom}(X,b):\uu{Hom}(X,Y)\sr \uu{Hom}(X,Y')$$
such that for all $Y$ one has
$$\uu{Hom}(X,Id_Y)=Id_{\uu{Hom}(X,Y)}$$
and for all $b:Y\sr Y'$, $b':Y'\sr Y''$ one has
$$\uu{Hom}(X,b\circ b')=\uu{Hom}(X,b)\circ \uu{Hom}(X,b')$$
\item For all $X,Y$ a morphism 
$$ev^X_Y:\uu{Hom}(X,Y)\times X\sr Y$$
such that for all $W$ the function
$$adj^{W,X}_{Y}:Mor(W,\uu{Hom}(X,Y))\sr Mor(W\times X, Y)$$
given by 
\begin{eq}
\llabel{2016.11.28.eq2}
u\mapsto (u\times Id_{X})\circ ev^X_Y
\end{eq}
is a bijection and such that for all $b:Y\sr Y'$ the square 
\begin{eq}
\llabel{2016.11.28.eq1}
\begin{CD}
\uu{Hom}(X,Y)\times X @>ev^X_Y>> Y\\
@V\uu{Hom}(X,b)\times Id_XVV @VVbV\\
\uu{Hom}(X,Y')\times X @>ev^X_{Y'}>> Y'
\end{CD}
\end{eq}
commutes. 
\end{enumerate}
A cartesian closed category is a category together with a cartesian closed structure on it.
\end{definition}
By definition the objects $\uu{Hom}(X,Y)$ are functorial only in $Y$. Their functoriality in $X$ is a consequence of a lemma. For $X,X',Y$ and $a:X\sr X'$ let 
$$\uu{Hom}(a,Y):\uu{Hom}(X',Y)\sr \uu{Hom}(X,Y)$$
be the unique morphism such that 
\begin{eq}
\llabel{2016.11.28.eq5}
adj(\uu{Hom}(a,Y))=(Id_{\uu{Hom}(X',Y)}\times a)\circ ev^{X}_{Y}
\end{eq}
Then one has:
\begin{lemma}
\llabel{2015.04.10.l1}
The morphisms $\uu{Hom}(-,Y)$ satisfy the equalites
$$\uu{Hom}(a\circ a',Y)=\uu{Hom}(a',Y)\circ \uu{Hom}(a,Y)$$
$$\uu{Hom}(Id_X,Y)=Id_{\uu{Hom}(X,Y)}$$
making $\uu{Hom}(-,Y)$ into a contravariant functor from ${\cal C}$ to itself.

In addition, for all $b:Y\sr Y'$ the square
$$
\begin{CD}
\uu{Hom}(X',Y) @>\uu{Hom}(X,b)>> \uu{Hom}(X',Y')\\
@V\uu{Hom}(a,Y)) VV @VV\uu{Hom}(a,Y') V\\
\uu{Hom}(X,Y) @>\uu{Hom}(X',b)>> \uu{Hom}(X,Y')
\end{CD}
$$
commutes.
\end{lemma}
\begin{proof}
It is a particular case of \cite[Theorem 3, p.100]{MacLane}. The commutativity of the square is a part of the "bifunctor" claim of the theorem. 
\end{proof}
\begin{lemma}
\llabel{2015.04.20.l2}
In a cartesian closed category let $X,X',Y$ be objects and let $a:X\sr X'$ be a morphism. Then the square
$$
\begin{CD}
\uu{Hom}(X',Y)\times X  @>Id_{\uu{Hom}(X',Y)}\times a>>  \uu{Hom}(X',Y)\times X' \\
@V\uu{Hom}(a,Y)\times Id_{X}VV @VVev^{X'}_Y V\\
\uu{Hom}(X,Y)\times X  @>ev^X_Y>> Y
\end{CD}
$$
commutes.
\end{lemma}
\begin{proof}
Let us show that both paths in the square are adjoints to $\uu{Hom}(a,Y)$. For the path that goes through the upper right corner it follows from the definition of  $\uu{Hom}(a,Y)$ as the morphism whose adjoint is $(Id\times a)\circ ev$. For the path that goes through the lower left corner it follows from the definition of adjoint applied to $\uu{Hom}(a,Y)$. Indeed, the adjoint to this morphism is
$$adj(\uu{Hom}(a,Y))=(\uu{Hom}(a,Y)\times Id_{X})\circ ev^X_Y$$
\end{proof}
\begin{lemma}
\llabel{2015.05.12.l2}
Let $\cal C$ be a cartesian closed category. Let $X,Y,W\in {\cal C}$, then one has: 
\begin{enumerate}
\item Let $b:Y\sr Y'$ be a morphism. Then for any $r\in Mor(W,\uu{Hom}(X,Y'))$ one has
$$adj(r\circ \uu{Hom}(X,b))=adj(r)\circ b$$
\item Let $a:X\sr X'$ be a morphism. Then for any $r\in Mor(W,\uu{Hom}(X',Y))$  one has
$$adj(r\circ \uu{Hom}(a,Y))=(Id_{W}\times a)\circ adj(r)$$
\item Let $c:W\sr W'$ be a morphism. Then for any $r\in Mor(W',\uu{Hom}(X,Y))$ one has
$$adj(c\circ r)=(c\times Id_{X})\circ adj(r)$$
\end{enumerate}
\end{lemma}
\begin{proof}
The proof of the first case is given by 
$$adj(r\circ \uu{Hom}(X,b))=
((r\circ \uu{Hom}(X,b))\times Id_X)\circ ev^{X}_{Y}=
$$$$(r\times Id_X)\circ (\uu{Hom}(X,b))\times Id_X)\circ ev^{X}_{Y}=
$$$$(r\times Id_X)\circ ev^{X}_{Y'}\circ b=
adj(r)\circ b$$
where the first equality is by (\ref{2016.11.28.eq2}), second equality by Lemma \ref{2015.05.14.l1}, the third equality by the commutativity of (\ref{2016.11.28.eq1}) and the fourth equality again by  (\ref{2016.11.28.eq2}).

The proof of the second case is given by the following sequence of equalities where we use the notation $Hm$ for $\uu{Hom}(a,Y)$ as well as a number of other abbreviations:
$$adj(r\circ Hm)=
((r\circ Hm)\times Id)\circ ev =
(r\times Id)\circ (Hm\times Id)\circ ev=
(r\times Id)\circ adj(Hm)=
$$$$(r\times Id)\circ (Id\times a)\circ ev=
(r\times a)\circ ev=
(Id\times a)\circ (r\times Id)\circ ev=
(Id\times a)\circ adj(r)$$
where the first equality is by (\ref{2016.11.28.eq2}), the second by (\ref{2016.11.28.eq3}), the third by (\ref{2016.11.28.eq2}), the fourth by (\ref{2016.11.28.eq5}), the fifth by (\ref{2016.11.26.eq2}), the sixth by (\ref{2016.11.26.eq2}) and the seventh by (\ref{2016.11.28.eq2}).

The proof of the third case is given by
$$adj(c\circ r)=
((c\circ r)\times Id_X) \circ ev^{X}_{Y}=
(c\times Id_X)\circ (r\times Id_X)\circ ev^{X}_{Y}=
$$$$(c\times Id_X)\circ adj(r)$$
where the first equality is by (\ref{2016.11.28.eq2}), second equality by (\ref{2016.11.28.eq3}) and the third equality by (\ref{2016.11.28.eq2}). 

Lemma is proved. 
\end{proof}

\subsection{Slice categories, pullbacks and locally cartesian closed categories}
\label{App.2}

For a category $\cal C$ and $Z\in {\cal C}$ one denotes by ${\cal C}/Z$ the slice category of $\cal C$ over $Z$. When one works in set theory one has to choose one of the several possible definitions of ${\cal C}/Z$. Indeed, the set of objects of ${\cal C}/Z$ can be defined as the set of pairs $(X,f)$ where $X\in {\cal C}$ and $f:X\sr Z$ or as the set of morphisms $f\in Mor(C)$ such that $codom(f)=Z$. There is an obvious bijection between these two sets but they are not equal. We define $Ob({\cal C}/Z)$ as the set of pairs $(X,f)$. Even more choices exist in the definition of the set of morphisms of ${\cal C}/Z$. One definition is the set of triples $(((X,f),(Y,g)),a)$ where $(X,f),(Y,g)\in Ob({\cal C}/Z)$ and $a:X\sr Y$ is such that $f=a\circ g$. Another one is the set of pairs $(a,g)$ where $a,g\in Mor(C)$ are such that $codom(a)=dom(g)$ and $codom(f)=Z$. Again, these sets are obviously isomorphic but not equal. Various other choices are possible. We will use the second option. We denote the pair $(a,g)$ by $a^g$. 

The mappings $(X,f)\mapsto X$ and $a^g\mapsto a$ define a functor ${\cal C}/Z\sr {\cal C}$ that we denote by $\pi_{Z,\#}$. We will rarely write the functions $(\pi_{Z,\#})_{Ob}$ and $(\pi_{Z,\#})_{Mor}$ explicitly using them instead as ``coercions''. Formally speaking, we will assume that $(\pi_{Z,\#})_{Ob}$ (resp. $(\pi_{Z,\#})_{Mor}$) is inserted in our notation whenever an object (resp. a morphism) of ${\cal C}/Z$ is specified where an object (resp. a morphism) of ${\cal C}$ is required.

We will say that $a:X\sr Y$ is a morphism over $Z$ if $a\circ g=f$. For given $(X,f)$ and $(Y,g)$, the function 
\begin{eq}
\llabel{2016.11.26.eq3}
a^g\mapsto a
\end{eq}
defines a bijection between morphisms $(X,f)\sr (Y,g)$ in ${\cal C}/Z$ and morphisms $X\sr Y$ over $Z$ in $\cal C$. 

In a category with binary products the morphism $Id_{Z}\times b$ satisfies the equality
$$(Id_{Z}\times b)\circ pr_1=pr_1$$
and therefore defines a morphism from $(Z\times Y,pr_1)$ to $(Z\times Y',pr_1)$ in ${\cal C}/Z$. We will denote this morphism in the slice category by $Z\times b$. Since (\ref{2016.11.26.eq3}) is injective, the equalities (\ref{2016.11.26.eq1}) and (\ref{2016.11.26.eq2}) imply that
\begin{eq}
\llabel{2016.11.30.eq1}
Z\times Id_{Y}=Id_{(Z\times Y,pr_1)}
\end{eq}
and
\begin{eq}
\llabel{2016.11.30.eq2}
Z\times (b\circ b')=(Z\times b)\circ (Z\times b')
\end{eq}
that is, that the mappings $X\mapsto (Z\times Y,pr_1)$, $b\mapsto Z\times b$ define a functor $Z\times -$ from ${\cal C}$ to ${\cal C}/Z$.  

The same holds for morphisms of the form $a:X\sr X'$. We denote the morphism in ${\cal C}/Z$ corresponding to the morphism $a\times Id_Z$ by $a\times Z$ and the resulting functor ${\cal C}\sr {\cal C}/Z$ by $-\times Z$. 
\begin{lemma}
\llabel{2016.12.16.l1}
Let 
\begin{eq}
\llabel{2016.12.16.eq1}
\begin{CD}
X @>a>> Y\\
@Va' VV @VVg V\\
Y' @>g'>> Z
\end{CD}
\end{eq}
be a commutative square of morphisms in $\cal C$ and $f=a\circ g=a'\circ g'$. Then 
\begin{eq}
\llabel{2016.12.16.eq2b}
\begin{CD}
(X,f) @>a^g>> (Y,g)\\
@V(a')^{g'}VV\\
(Y,g') @.
\end{CD}
\end{eq}
is a binary product diagram in ${\cal C}/Z$ if and only if (\ref{2016.12.16.eq1}) is a pullback in $\cal C$.
\end{lemma}
\begin{proof}
Assume that (\ref{2016.12.16.eq2b}) is a binary product diagram. Let $W\in {\cal C}$ and let $d:W\sr Y$, $d':W\sr Y'$ be such that $d\circ g=d'\circ g'$. Let $e=d\circ g$. Then $d^g:(W,e)\sr (Y,g)$ and $(d')^{g'}:(W,e)\sr (Y',g')$ are morphisms in ${\cal C}/Z$ and therefore there exists $c^f:(W,e)\sr (X,f)$ such that $c^f\circ a^g=d^g$ and $c^f\circ (a')^{g'}=(d')^{g'}$ in ${\cal C}/G$, that is, $c\circ a=d$ and $c\circ a'=d'$ in $\cal C$. Let $c':W\sr X$ be another morphism in ${\cal C}$ such that $c'\circ a=d$ and $c'\circ a'=d'$. Then $e=d\circ g=c'\circ a\circ g=c'\circ f$ and therefore $(c')^f$ is a morphism $(W,e)\sr (X,f)$ in ${\cal C}/Z$. Next, $(c')^f\circ a^g=(c'\circ a)^g=d^g$ and $(c')^f\circ (a')^{g'}=(c'\circ a')^{g'}=(d')^{g'}$. Therefore $(c')^f=c^f$, that is, $c=c'$. This shows that (\ref{2016.12.16.eq1}) is a pullback in $\cal C$.

Similar reasoning shows that if (\ref{2016.12.16.eq1}) is a pullback in $\cal C$ then (\ref{2016.12.16.eq2b}) is a binary product diagram in ${\cal C}/Z$.
\end{proof}

Lemma \ref{2016.12.16.l1}, combined with a related statement about commutative squares, implies that a choice of binary product structures on all the slice categories ${\cal C}/Z$ is ``the same as'' the choice of pullbacks for all pairs of morphisms with the common codomain in $\cal C$. 

To be precise we have to say that how to construct a bijection between the set of families of binary product structures on the categories ${\cal C}/Z$ for all $Z$ and the set of pullback structures on ${\cal C}$. 

We usually denote the distinguished binary product of $(X,f)$ and $(Y,g)$ in ${\cal C}/Z$ by $(X,f)\times_Z (Y,g)$ and the canonical morphism from $(X,f)\times_Z (Y,g)$ to $Z$ by $f\dd g$. 

For $f:X\sr Z$ and $g:Y\sr Z$, the two commutative triangles formed by $pr_1:(X,f)\times_Z(Y,g)\sr (X,f)$, $f$, $f\dd g$ and $pr_2:(X,f)\times_Z(Y,g)\sr (Y,g)$, $g$, $f\dd g$ are adjacent and define the familiar commutative square of the pullback of $f$ and $g$. 

This defines a function in one direction.

For $f:X\sr Z$ and $g:Y\sr Z$, the diagonal of the pullback square based on $f$ and $g$ is an object over $Z$ and the two projections define morphisms from this object to $(X,f)$ and $(Y,g)$ respectively. The corresponding pair of morphisms in ${\cal C}/Z$ is a binary product diagram. This defines a morphism in the other direction. 

The fact that these morphisms are inverse to each other follows readily from the construction. 


Given a binary products structure on ${\cal C}/Z$, morphisms $f:X\sr Z$, $g:Y\sr Z$ and morphisms $a:X'\sr X$, $b:Y'\sr Y$ we have a morphism $a^f\times_Z b^g$
which is the unique morphism in ${\cal C}/Z$ of the form
$$a^f\times_Z b^g:(X',a\circ f)\times_Z(Y',b\circ g)\sr (X,f)\times_Z(Y,g)$$
such that
\begin{eq}
\llabel{2016.11.24.eq1}
(a^f\times_Z b^g)\circ pr_1=pr_1\circ a^f
\end{eq}
and
\begin{eq}
\llabel{2016.11.24.eq2}
(a^f\times_Z b^g)\circ pr_2=pr_2\circ b^g
\end{eq}
\begin{lemma}
\llabel{2015.05.14.l1}
In the setting introduced above one has:
\begin{enumerate}
\item $Id_{(X,f)\times_Z(Y,g)}=Id_{(X,f)}\times_Z Id_{(Y,g)}$,
\item suppose that we have in addition morphisms $a':X''\sr X'$ and $b':Y''\sr Y'$. Then
$$((a')^{a\circ f}\times_Z (b')^{b\circ g})\circ (a^f\times_Z b^g)=(a'\circ a)^f\times_Z(b'\circ b)^g$$
\end{enumerate}
\end{lemma}
\begin{proof}
It is a particular case of (\ref{2016.11.26.eq1}) and (\ref{2016.11.26.eq2}).
\end{proof}
Following the general case considered in Appendix \ref{App.1} we will write $(X,f)\times_Z b^{g}$ (resp. $a^f\times_Z (Y,g)$) for the morphism in ${\cal C}/X$ (resp. ${\cal C}/Y$) corresponding to $Id_{(X,f)}\times_Z b^g$ (resp. $a^f\times_Z Id_{(Y,g)}$). 

In view of Lemma \ref{2015.05.14.l1} and (\ref{2016.11.24.eq1}), for any $(X,f:X\sr Z)$, the functions
$$(Y,g)\mapsto ((X,f)\times_Z(Y,g),pr_1)$$
$$(b^g:(Y',g')\sr (Y,g))\mapsto (X,f)\times_Z b^g$$
form a functor from ${\cal C}/Z$ to ${\cal C}/X$ and similarly by Lemma \ref{2015.05.14.l1} and (\ref{2016.11.24.eq2}), for any $(Y,g:Y\sr Z)$ the functions 
$$(X,f)\mapsto ((X,f)\times_Z(Y,g), pr_2)$$
$$(a^f:(X',f')\sr (X,f))\mapsto a^f\times_Z (Y,g)^{f}$$
form a functor from ${\cal C}/Z$ to ${\cal C}/Y$.
\begin{definition}
\llabel{2015.03.27.def1}
A locally cartesian closed structure on a category $\cal C$ is a family of (binary) cartesian closed structures on the categories ${\cal C}/Z$ for all $Z\in {\cal C}$.

We usually denote the binary product on ${\cal C}/Z$ as above. 

We usually denote the internal-hom objects in ${\cal C}/Z$ by $\uu{Hom}_Z((X,f),(Y,g))$ and the canonical morphisms from $\uu{Hom}_Z((X,f),(Y,g))$ to $Z$ by $f \triangle g$. 

The rest of the notations ($\uu{Hom}_Z((X,f),b^g)$, $ev^{(X,f)}_{(Y,g)}$, $adj^{(W,h),(X,f)}_{(Y,g)}$, $\uu{Hom}_Z(a^f,(Y,g))$) immediately follow from the ones introduced previously.

A locally cartesian closed category is a category together with a locally cartesian closed structure on it.
\end{definition}
The name ``locally cartesian closed'' follows naturally from this definition and the intuition based on the example of the category of open sets of a topological space or a Grothendieck site. If only the subsets of the open sets of a particular covering are known then one sometimes says that the space is known only locally, but the global structure that arises from gluing of all these subsets together is not known. Hence the ``local'' structure of a category is given by the structure of its slice categories. 
\begin{example}\rm
\llabel{2015.05.20.ex1}
The following example shows that there can be many different structures of a category with pullbacks on a category and also many locally cartesian closed structures.

Let us take as our category the category $F$ whose objects are natural numbers and 
$$Mor(n,m)=Fun(\{0,\dots,n-1\},\{0,\dots,m-1\})$$  

Since every isomorphism class contains exactly one object every auto-equivalence of this category is an automorphism. Let $\Phi$ be such an automorphism. It is easy to see that it must be identity on the set of objects. Let $X=\{0,1\}$. Consider $\Phi$ on $End(X)$. Since $\Phi$ must respect identities and compositions, $\Phi$ must take $Aut(X)$ to itself and must act on it by identity. If $1$ and $\sigma$ are the two elements of $Aut(X)$ we conclude that $\Phi(1)=1$ and $\Phi(\sigma)=\sigma$. 

Let us choose now any structure of a category with pullbacks on $F$ and let us consider two new structures $str_1$ and $str_{\sigma}$ that are obtained by modifying pullbacks as follows. In both structures we set all pullbacks to be as they were except for the pullback of the pair of morphisms $(Id_X,Id_X)$. For this pair we set the pullbacks to be as follows:
\begin{eq}\llabel{2015.05.20.sq1}
\begin{CD}
X @>Id_X>> X\\
@V Id_X VV @VV Id_X V\\
X @>Id_X>> X
\end{CD}
\spc\spc\mbox{\rm for $str_1$ and}\spc\spc
\begin{CD}
X @>\sigma>> X\\
@V \sigma VV @VV Id_X V\\
X @>Id_X>> X
\end{CD}
\spc\spc\mbox{\rm for $str_{\sigma}.$}
\end{eq}
The preceding discussion shows that there is no auto-equivalence which would transform $str_1$ into $str_{\sigma}$. 

The category $F$ also has a locally cartesian closed structure and it can be shown that it can be modified so that its pullback components are $str_1$ and $str_{\sigma}$. This shows that $F$ has at least two locally cartesian closed structures that are not equivalent modulo the auto-equivalences of $F$.

The solution to this seeming paradox is that there is a category structure on the set of pullback structures (resp. locally cartesian closed structures) on a category. Any two pullback structures (resp. lcc structures) are isomorphic in this category and in this sense pullbacks on a category are ``unique''. 
\end{example}
\begin{remark}\rm
\llabel{2015.05.20.rem1}
The following remark is not used anywhere in the paper and we include it only as a comment for those readers who are familiar with the univalent foundations where there is a notion of a category and pre-category. There the types of pullback structures and of locally cartesian closed structures on a category (as opposed to those on a general pre-category) are of h-level 1, i.e., classically speaking are either empty or contain only one element. 

In addition any such structure on a pre-category should define a structure of the same kind on the Rezk completion of this pre-category with all the different structures on the pre-category becoming equal on the Rezk completion. In the case of the previous example the Rezk completion of $F$ is the category $FSets$ of finite sets and in view of the univalence axiom for finite sets the two pullbacks of \ref{2015.05.20.sq1} will become equal in $FSets$. 
\end{remark}


\section{Acknowledgements}
\label{Sec.Ack}

I am grateful to the Department of Computer Science and Engineering of the University of Gothenburg and Chalmers University of Technology for its the hospitality during my work on the first version of the paper.  

Work on this paper was supported by NSF grant 1100938.

This material is based on research sponsored by The United States Air Force Research Laboratory under agreement number FA9550-15-1-0053. The US Government is authorized to reproduce and distribute reprints for Governmental purposes notwithstanding  any copyright notation thereon.

The views and conclusions contained herein are those of the author and should not be interpreted as necessarily representing the official policies or endorsements, either expressed or implied, of the United States Air Force Research Laboratory, the U.S. Government or Carnegie Mellon University.

\def\cprime{$'$}


\begin{thebibliography}{10}

\bibitem{Cartmell0}
John Cartmell.
\newblock Generalised algebraic theories and contextual categories.
\newblock {\em Ph.D. Thesis, Oxford University}, 1978.
\newblock \url{http://www.cs.ru.nl/~spitters/Cartmell.pdf}.

\bibitem{Cartmell1}
John Cartmell.
\newblock Generalised algebraic theories and contextual categories.
\newblock {\em Ann. Pure Appl. Logic}, 32(3):209--243, 1986.

\bibitem{GambinoHyland}
Nicola Gambino and Martin Hyland.
\newblock Wellfounded trees and dependent polynomial functors.
\newblock In {\em Types for proofs and programs}, volume 3085 of {\em Lecture
  Notes in Comput. Sci.}, pages 210--225. Springer, Berlin, 2004.

\bibitem{MacLane}
S.~MacLane.
\newblock {\em Categories for the working mathematician}, volume~5 of {\em
  Graduate texts in {M}athematics}.
\newblock Springer-Verlag, 1971.

\bibitem{Streicher}
Thomas Streicher.
\newblock {\em Semantics of type theory}.
\newblock Progress in Theoretical Computer Science. Birkh\"auser Boston Inc.,
  Boston, MA, 1991.
\newblock Correctness, completeness and independence results, With a foreword
  by Martin Wirsing.

\bibitem{Cfromauniverse}
Vladimir Voevodsky.
\newblock A {C}-system defined by a universe category.
\newblock {\em Theory Appl. Categ.}, 30(37):1181--1215, 2015.
\newblock \url{http://www.tac.mta.ca/tac/volumes/30/37/30-37.pdf}.

\bibitem{fromunivwithPi}
Vladimir Voevodsky.
\newblock Products of families of types in the {C-systems} defined by a
  universe category.
\newblock {\em arXiv 1503.07072}, pages 1--30, 2015.

\bibitem{fromunivwithPiI}
Vladimir Voevodsky.
\newblock Products of families of types and {$(\Pi,\lambda)$}-structures on
  {C}-systems.
\newblock {\em Theory Appl. Categ.}, 31(36):1044--1094, 2016.
\newblock \url{http://www.tac.mta.ca/tac/volumes/31/36/31-36.pdf}.

\bibitem{Csubsystems}
Vladimir Voevodsky.
\newblock Subsystems and regular quotients of {C-systems}.
\newblock In {\em Conference on Mathematics and its Applications, (Kuwait City,
  2014)}, number 658 in Contemporary Mathematics, pages 127--137, 2016.

\bibitem{fromunivwithPiII}
Vladimir Voevodsky.
\newblock The {$(\Pi,\lambda)$}-structures on the {C}-systems defined by
  universe categories.
\newblock {\em In preparation}, 2017.

\end{thebibliography}

\end{document}